\documentclass[11pt]{amsart}
\newcommand{\ds}{\displaystyle}
\newcommand{\ben}{\begin{enumerate}}
\newcommand{\een}{\end{enumerate}}
\newcommand{\eq}[2][label]{\begin{equation}\label{#1}#2\end{equation}}
\newcommand{\av}[2]{\langle #1\rangle_{_{\scriptstyle #2}}}
\newcommand{\ave}[1]{\langle #1\rangle}
\newcommand{\bav}[2]{\big\langle #1\big\rangle_{_{\scriptstyle #2}}}

\newcommand{\sav}[2]{\langle #1\rangle_{_{\scriptscriptstyle #2}}}

\DeclareMathOperator\sign{sgn}

\newcommand{\AC}{A_\infty^C(Q)}

\newcommand{\ve}{\varepsilon}
\usepackage{amsmath}
\usepackage{amsfonts}
\usepackage{amssymb}
\usepackage{amsthm}
\usepackage{graphicx,color,hyperref}

\usepackage{enumitem}

\newcommand{\bel}[1]{\boldsymbol{#1}}

\newcommand{\ma}{Monge--Amp\`{e}re }

\newcommand{\BMO}{{\rm BMO}}

\def\sgn{\operatorname{sgn}}
\newcommand{\rn}{\mathbb{R}^n}

\newtheorem{theorem}{Theorem}[section]

\newtheorem{lemma}[theorem]{Lemma}
\newtheorem{corollary}[theorem]{Corollary}

\newtheorem*{theorem*}{Theorem}{\bf}{\it}
\newtheorem*{proposition*}{Proposition}{\bf}{\it}
\newtheorem*{observation*}{Observation}{\bf}{\it}
\newtheorem*{lemma*}{Lemma}{\bf}{\it}

\theoremstyle{definition}
\newtheorem{definition}[theorem]{Definition}

\theoremstyle{remark}
\newtheorem{remark}[theorem]{Remark}

\numberwithin{equation}{section}

\allowdisplaybreaks
\begin{document}

\title{The John--Nirenberg constant of $\BMO^p,$ $1\le p\le 2$}

\author{Leonid Slavin}
\address{University of Cincinnati}
\email{leonid.slavin@uc.edu}

\thanks{L. Slavin's research supported in part by the NSF (DMS-1041763)}



\begin{abstract}
We compute the exact John--Nirenberg constant of $\BMO^p\big((0,1)\big)$ for $1\le p\le 2,$ which has been known only for $p=1$ and $p=2.$ We also show that this constant is attained in the weak-type John--Nirenberg inequality and obtain a sharp lower estimate for the distance in $\BMO^p$ to $L^\infty.$ These results rely on sharp $L^p$- and weak-type estimates for logarithms of $A_\infty$ weights, which in turn use the exact expressions for the corresponding Bellman functions.
\end{abstract}

\maketitle
\section{Preliminaries and main results}
\label{main_results}
By $\av{\varphi}Q$ we denote the average of a locally integrable function over a cube $Q\subset\rn$ with respect to the Lebesgue measure:
$$
\av{\varphi}Q=\frac1{|Q|}\int_Q\varphi.
$$
Fix a $p>0$ and let $\BMO^p$ be the (factor-)space
\eq[1]{
\BMO^p=\{\varphi\in L^1_{loc}\colon\quad\|\varphi\|_{\BMO^p}:=\sup_{\text{cube}~J}\av{|\varphi-\av{\varphi}J|^p}J^{1/p}<\infty\}.
}
If the supremum is taken over all cubes $J\subset\rn,$ we write $\BMO^p(\rn);$ if the supremum is over all subcubes $J$ of a fixed cube $Q,$ we write $\BMO^p(Q).$ If the context is clear or the statement applies to both cases, we write simply $\BMO^p.$ It is known that all $p$-based (quasi-)norms are equivalent, meaning that~\eqref{1} defines the same space for all $p>0.$ Thus, we can, and will, write simply $\BMO$ in the left-hand side of~\eqref{1}, keeping the symbol $\BMO^p$ for the norm.

A weight is an almost everywhere positive function. We say that a weight $w$ belongs to $A_\infty,$ $w\in A_\infty,$ if both $w$ and $\log w$ are locally integrable and the following condition holds:
\eq[1.1]{
[w]_{A_\infty}:=\sup_{\text{cube~}J}\av{w}Je^{-\av{\log w}J}<\infty.
}
We say that $w$ belongs to $A_2,$ $w\in A_2,$ if both $w$ and $w^{-1}$ are locally integrable and  
\eq[1.2]{
[w]_{A_2}:=\sup_{\text{cube~}J}\av{w}J\av{w^{-1}}J<\infty.
}
The quantities $[w]_{A_\infty}$ and $[w]_{A_2}$ are called the $A_\infty$- and the $A_2$-characteristics of $w,$ respectively.  As before,  we write $A_\infty(\rn),$ $A_\infty(Q),$ or simply $A_\infty,$ as appropriate, and similarly for $A_2.$ Clearly, $A_2\subset A_\infty.$ In fact, it is easy to show that $w$ is in $A_2$ if and only if both $w$ and $w^{-1}$ are in $A_\infty.$

As the title suggests, our headline results are for $\BMO^p.$ In large part, they arise as consequences of sharp estimates for logarithms of $A_\infty$ weights. In turn, these estimates rely on explicit Bellman functions, and that material is taken up in the next section. Here we first state the $\BMO^p$ results, then the $A_\infty$ results, and then explain how they are connected. 

Throughout the paper, we will use the following notation. Fix $C\ge1$ and let $\xi^\pm=\xi^\pm(C)$ be the two solutions of the equation 
\eq[xi2]{
e^{-\xi}=C(1-\xi):\quad-\infty<\xi^-\le0\le\xi^+<1.
}
Note that $\xi^\pm(1)=0$ and that $\xi^+$ is strictly increasing with $\lim_{C\to\infty}\xi^+(C)=1,$ while $\xi^-$ is strictly decreasing with $\lim_{C\to\infty}\xi^-(C)=-\infty.$

For $C>1,$ let
\eq[kc]{
k(C)=2\,\frac{(1-\xi^-)(1-\xi^+)}{\xi^+-\xi^-}\,(C-1)%=\frac2{e^{\xi^+}-e^{\xi^-}}\Big(1-\frac1C\Big)
}
and set $k(1)=0.$ It is a simple exercise to verify that $k$ is continuous on $[1,\infty)$ and that $\lim_{C\to\infty}k(C)=\frac2e.$ It is a little harder but still straightforward to check that $k$ is also strictly increasing. 

\subsection{Main results for $\BMO^p$}
The key fact about $\BMO$ is the John--Nirenberg inequality, originally proved in~\cite{jn}. We will need it in two forms:
\medskip

\noindent{\bf Weak-type form}: There are positive constants $c_0(p)$ and $K(p)$ such that if $\varphi\in\BMO,$ then for any cube $J$ and any $\lambda\ge0,$
\eq[jn1]{
\frac1{|J|}|\{t\in J: |\varphi(t)-\av{\varphi}J|\ge\lambda\}|\le K(p)e^{-c_0(p)\lambda/\|\varphi\|_{\scriptscriptstyle \BMO^p}}.
}
\noindent{\bf Integral form}: There exists $\ve_0(p)>0$ such that if $\varphi\in\BMO$ with $\|\varphi\|_{\scriptscriptstyle\BMO^p}\le\ve<\ve_0(p),$ then for any cube $J,$
\eq[jn2]{
\av{e^\varphi}J \le C(\ve,p)e^{\av{\varphi}{\scriptscriptstyle J}},
}
for some function $C$ of $\ve$ and $p.$

Of course, all these constants also depend on dimension, but most of our results are in dimension 1, so we will suppress that dependence. Importantly, and somewhat surprisingly, they also appear to depend on whether one considers $\BMO(\rn)$ or the local variant $\BMO(Q)$ (though clearly not on the exact cube $Q$ chosen, so one can always take $Q=(0,1)^n$). Let us reserve the names $\ve_0(p)$ and $C(\ve,p)$ for the best constants in~\eqref{jn2} in the case when $\BMO=\BMO((0,1)).$ We call $\ve_0(p)$ the John--Nirenberg constant of $\BMO^p((0,1)).$
In two instances below we will need to differentiate $\ve_0(p)$ from its counterpart for the line, which we will denote by $\ve^\mathbb{R}_0(p).$ It is easy to see that $\ve_0^\mathbb{R}(p)\ge\ve_0(p).$

It is a classical consequence of Gehring's theorem on self-improvement of reverse H\"older classes \cite{gehring} that $A_\infty$ weights self-improve and hence $\ve_0(p)$ is not attained in~\eqref{jn2} (see Theorem~\ref{t7} below for a sharp version of this result). Also, clearly
\eq[epsc]{
\ve_0(p)=\sup\{c_0(p): \eqref{jn1}\text{~holds for some}~K(p)\}.
}
However, it is far from clear whether $\ve_0(p)$ is attained as $c_0(p)$ in~\eqref{jn1} or what the best $K(p)$ might be if it is attained.

Observe that~\eqref{jn2} means that if $\varphi\in\BMO,$ then $e^{\ve\varphi}\in A_\infty$ for all sufficiently small $\ve>0.$ For $\varphi\in\BMO,$ let
\eq[ephi]{
\ve_\varphi=\sup\{\ve>0:~e^{\ve \varphi}\in A_\infty\}
}
(note that $\ve_\varphi$ can be infinity). Directly from the definitions of $\ve_0(p)$ and $\ve_{\varphi}$ we have
$$
\ve_0(p)\le\inf\{ \ve_\varphi>0:~\|\varphi\|_{\BMO^p}\!=\!1\}=\inf\{\ve>0:~\forall\varphi~\|\varphi\|_{\BMO^p}\!=\!1~\Longrightarrow ~e^{\ve\varphi}\in A_2\}.
$$
The inequality on the left is obvious; we actually show in these pages that it holds as equality, at least for a range of $p.$ The identity on the right holds because the set $\{ \ve_\varphi>0:~\|\varphi\|_{\BMO^p}\!=\!1\}$ contains both $\ve_\varphi$ and $\ve_{-\varphi}$ for each $\varphi$ with $\|\varphi\|_{\BMO^p}\!=\!1.$

We have two main tasks: to determine the exact value of $\ve_0(p)$ and to show that $\ve_0(p)$ is attained as $c_0(p)$ in~\eqref{jn1}. Currently, our solutions to these problems are limited to the range $1\le p\le2$ for the first task, and $1<p\le 2$ for the second. The only known results up to this point have been those for $p=1$ and $p=2.$ Specifically, Korenovskii \cite{korenovsky} found $\ve_0(1)=\frac2e$ and showed that it is attained in~\eqref{jn1}; Lerner \cite{lerner} showed that the best $K(1)$ in~\eqref{jn1} for $c_0(1)=\frac2e$ is $\frac12e^{4/e}.$ Vasyunin and the author \cite{sv} found $\ve_0(2)=1$ and $C(\ve,2)=\frac{e^{-\ve}}{1-\ve};$ and Vasyunin and Volberg \cite{vv} showed that $c_0(2)=1$ is attained in~\eqref{jn1} with the best $K(2)=e.$ 

Here is our main theorem for $\BMO^p.$ 
\begin{theorem}
\label{t1}
For $p\in[1,2]$,
\eq[maineps]{
\ve_0(p)=\left[\frac pe\Big(\Gamma(p)-\int_0^1t^{p-1}e^t\,dt\Big)+1\right]^{1/p}.
}
Furthermore,
if $1<p\le 2,$ then for all $(2-p)\ve_0(p)\le \ve<\ve_0(p),$
\eq[cep]{
C(\ve,p)=\frac{e^{-\ve/\ve_0(p)}}{1-\ve/\ve_0(p)};
}
and for all $0\le\ve < \frac{2}e,$  
\eq[ce1]{
\frac{e^{-\frac e2\ve}}{1-\frac e2\ve}\le C(\ve,1)\le k^{-1}(\ve),
}
where $k^{-1}:[0,\frac2e)\to[1,\infty)$ is the inverse function to $k.$
\end{theorem}

\begin{remark}
From~\eqref{kc}, we have $k(C)=\frac{2(1-1/C)}{e^{\xi^+}-e^{\xi^-}}\ge \frac 2e\,\big(1-\frac1C\big),$ thus $k^{-1}(\ve)\le \frac{1}{1-\frac e2\,\ve},$ and so
$$
C(\ve,1)\le \frac{1}{1-\frac e2\,\ve},\quad\text{if}\quad 0\le\ve<\frac{2}e.
$$
\end{remark}

Note that~\eqref{maineps} gives $\ve_0(1)=\frac2e,$ $\ve_0(2)=1,$ and $C(\ve,2)=\frac{e^{-\ve}}{1-\ve}$ for $\ve\in[0,1).$ Thus, we recover all known cases of sharp results in~\eqref{jn2}, but, of course, this theorem contains more. We do not know if the upper inequality in~\eqref{ce1} actually holds as equality, at least for some $\ve>0.$ This has to do with the nature of optimizers in our inequalities, and ultimately with the geometry of the underlying extremal problem; see Remark~\ref{rem6} below and the discussion in the next section.

Here is our weak-type result for $\BMO^p.$ 
\begin{theorem}
\label{t2}
If $p\in(1,2],$ then $\ve_0(p)$ is attained as $c_0(p)$ in~\eqref{jn1}. Specifically, if $Q$ is an interval and $\varphi\in\BMO(Q),$ then for any subinterval $J$ of $Q$ and any $\lambda\ge0,$
\eq[ttt2]{
\frac1{|J|}\,\big|\{t\in J:~|\varphi(t)-\av{\varphi}J|\ge\lambda\}\big|\le (p-1)^{-\frac1{2-p}} e^{-\frac{\ve_0(p)\lambda }{\|\varphi\|_{\BMO^p}}}.
}
\end{theorem}
Here we recover one of the two known cases in~\eqref{jn1}, namely, $p=2.$ We even get the exact constant $K(2)=e$ from~\cite{vv} (in the limit, as $p\to2^-$), though our focus is squarely on $\ve_0(p)$ and we certainly do not claim that the constant before the exponent is sharp for $p<2.$ In particular, $(p-1)^{-\frac1{2-p}}$ blows up as $p\to 1^+,$ and this theorem does not capture the case $p=1$ that was addressed in~\cite{korenovsky} and~\cite{lerner}.

We have two corollaries. The first relates to the famous theorem by Garnett and Jones~\cite{gj} that says that the distance from 
$\varphi\in\BMO(\rn)$ to $L^\infty$ in $\BMO^1$ norm admits two-sided estimates in terms of what we call $\ve_\varphi$ and $\ve_{-\varphi}.$ In our notation, their theorem is as follows:
\begin{theorem}[\cite{gj}]
If $\varphi\in\BMO(\rn),$ then there are constants $C_1$ and $C_2$ depending only of dimension such that
$$
\frac{C_1}{\min\{\ve_\varphi,\ve_{-\varphi}\}}\,\le\inf_{f\in L^\infty(\rn)}\|\varphi-f\|_{\BMO^1}\le\, \frac{C_2}{\min\{\ve_\varphi,\ve_{-\varphi}\}}.
$$
\end{theorem}
Of course, one can replace $\BMO^1$ norm with $\BMO^p$ norm and ask for the sharp constants $C_1(p), C_2(p).$ In that direction we have an interesting result that gives the sharp value of $C_1(p)$ in the case of $\BMO((0,1)),$ and then the same bound for $\BMO(\mathbb{R}),$ which may or may not be sharp.

\begin{corollary}
\label{t2.5}
If $p\in[1,2],$ $Q$ is an interval, and $\varphi\in\BMO(Q),$ then
\eq[dist1]{
\inf_{f\in L^\infty(Q)}\|\varphi-f\|_{\BMO^p(Q)}\ge \frac{\ve_0(p)}{\min\{\ve_{\varphi},\ve_{-\varphi}\}}.
}
Consequently, if $\varphi\in\BMO(\mathbb{R}),$ then
\eq[dist2]{
\inf_{f\in L^\infty(\mathbb{R})}\|\varphi-f\|_{\BMO^p(\mathbb{R})}\ge \frac{\ve_0(p)}{\min\{\ve_{\varphi},\ve_{-\varphi}\}}.
}
Inequality~\eqref{dist1} is sharp.
\end{corollary}
 \noindent We conjecture that replacing $\ve_0(p)$ with $\ve_0^\mathbb{R}(p)$ in~\eqref{dist2} makes that inequality sharp as well.
 
 This corollary allows us to estimate from below the $\BMO(\mathbb{R})\to L^\infty(\mathbb{R})$ norm of the Hilbert transform,
$$
Hf(x)=\frac1\pi\,{\rm p. v.}\int_\mathbb{R}\frac1{x-y}\,f(y)\,dy.
$$
The estimate is elementary, but it sheds light on what the actual value of that norm might be. The proof, which we give right away, uses the Helson--Szeg\"o theorem \cite{hs}, which says that $e^\varphi\in A_2(\mathbb{R})$ if and only if $\varphi=f_1+Hf_2,$ where $f_1,f_2\in L^\infty(\mathbb{R})$ with $\|f_2\|_{L^\infty}<\frac\pi2.$
\begin{corollary}
\label{2.8}
For $p\in[1,2],$
\eq[norm6]{
\|H\|_{L^\infty\to\BMO^p}\ge\frac2\pi\,\ve_0(p).
}
\end{corollary}
\begin{proof}
Take any $\varphi$ such that $\min\{\ve_{\varphi},\ve_{-\varphi}\}=1.$ Then for any $\ve\in(0,1),$ $e^{\ve\varphi}\in A_2.$ By the Helson--Szeg\"o theorem, 
$
\ve\varphi=f_1+Hf_2
$ with $f_1,f_2\in L^\infty$ and $\|f_2\|_\infty<\frac\pi2.$ Therefore, 
$$
\inf_{f\in L^\infty}\|\varphi-f\|_{\BMO^p}\le \frac1\ve\,\|Hf_2\|_{\BMO^p}\le \frac\pi{2\ve}\,\|H\|_{{L^\infty\to\BMO^p}}.
$$
By~\eqref{dist2} the left-hand side is no smaller than $\ve_0(p).$ Now, take the limit as $\ve\to1.$
\end{proof}
\noindent We again conjecture that~\eqref{norm6} holds with equality if $\ve_0(p)$ is replaced with $\ve_0^\mathbb{R}(p).$

\subsection{Main results for $A_\infty$}
We prove three theorems for $A_\infty.$ All three are interesting in their own right, but the the first two are also needed in the proof of Theorem~\ref{t2}. These results assume that $e^\varphi\in A_\infty(Q)$ for an interval $Q$ and give various estimates for $\varphi.$ Their short proofs, given in Section~\ref{nonB}, are based on the exact expressions of the corresponding Bellman functions. 

We first give lower estimates for $\|\varphi\|_{\BMO^p},$ $1\le p\le 2,$ and an upper estimate for $p=2.$ The lower inequalities for $p>1$ are required in the proof of Theorem~\ref{t3}.
\begin{theorem}
\label{t3}
Assume that $e^\varphi\in A_\infty(Q)$ and let $C=[e^\varphi]_{A_\infty(Q)}.$

For all $C\ge 1,$
\eq[est1]{
\ds \xi^+\le\|\varphi\|_{\BMO^2}\le|\xi^-|
}
and
\eq[est2]{
\|\varphi\|_{\BMO^1(Q)}\ge k(C).
}
If $p\in(1,2)$ and $C\ge\frac{e^{p-2}}{p-1},$ then
\eq[est3]{
\|\varphi\|_{\BMO^p(Q)}\ge \ve_0(p) \xi^+.
}
Inequalities~\eqref{est1} and~\eqref{est3} are sharp.
\end{theorem}
\begin{remark}
\label{rem6}
Inequality~\eqref{est2} may or may not be sharp. The corresponding inequality for oscillations, $\av{|\varphi-\av{\varphi}Q|}Q\ge k(C),$ is sharp, in that there exists a function $\varphi$ for which it becomes an inequality. However, that function's $\BMO^1$ norm is realized on an interval $J\subsetneq Q,$ and this norm is larger than $k(C).$ In contrast, for $p>1$ there is a different function $\varphi$ for which $\|\varphi\|_{\BMO^p(Q)}=\av{|\varphi-\av{\varphi}Q|^p}Q^{1/p}=\ve_0(p)\xi^+.$
\end{remark}

Note that in this theorem $C$ is precisely the $A_\infty$-characteristic of $e^\varphi.$ The next two results simply require that the characteristic be no larger than $C.$ If $Q$ is an interval and $C\ge1,$ we denote 
$$
\AC:=\{w\in A_\infty(Q):~ [w]_{A_\infty(Q)}\le C\}.
$$ 

The second theorem gives the sharp estimate for the distribution function of $\varphi$ when $e^\varphi\in \AC.$ We mention two previous results in this vein, both using Bellman functions: article \cite{vv} contains the sharp weak-type John--Nirenberg inequality for $\BMO^2;$ and article \cite{reznikov} gives the sharp inequalities for distribution functions of $A_{p_1,p_2}$ weights. Note, however, that our estimate is for logarithms of $A_\infty$ weights, not the weights themselves.

\begin{theorem}
\label{t8}
If $C\ge1$ and $e^\varphi\in \AC,$ then for any $\lambda\in\mathbb{R}$ and any subinterval $J$ of $Q,$
$$
\frac1{|J|}\,\big|\{t\in J:~\varphi(t)-\av{\varphi}J\ge\lambda\}\big|\le
\begin{cases}
1,& \lambda\le 0,
\bigskip

\\
\ds1-\frac{\lambda}{\xi^+-\xi^-},& 0\le \lambda \le -\xi^-,
\bigskip

\\
\ds\frac{\xi^+e^{-\xi^-/\xi^+}}{\xi^+-\xi^-}\,e^{-\lambda/\xi^+},& \lambda\ge -\xi^-.
\end{cases}
$$
This inequality is sharp for each value of of $\lambda.$
\end{theorem}
\begin{remark}
Observe that $\xi^+-\xi^--\lambda\le \xi^+e^{-(\xi^-+\lambda)/\xi^+}$ for $0\le\lambda\le -\xi^-.$ Thus, one has the following less precise, but more convenient estimate for all $\lambda\ge0:$
\eq[ones]{
\frac1{|J|}\,\big|\{t\in J:~\varphi(t)-\av{\varphi}J\ge\lambda\}\big|\le \frac{e^{-\xi^-/\xi^+}}{1-\xi^-/\xi^+}\,e^{-\frac\lambda{\xi^+}}.
}
\end{remark}

Finally, our third theorem for $A_\infty$ is a Gehring-type result on self-improvement of $A_\infty$ weights. It gives the sharp bound on $\delta>1$ such that if $e^\varphi\in A_\infty,$ then $e^{\delta\varphi}\in A_\infty.$ It is not needed in the proof of the main theorems for $\BMO^p,$ but it is a nice illustration of our method and involves hardly any additional effort, since we are already doing in-depth analysis on $A_\infty.$ While to our knowledge it is not formulated this way in literature, it is implicitly contained in~\cite{v1} and~\cite{v2}.
\begin{theorem}
\label{t7}
If $C\ge 1$ and $e^\varphi\in \AC,$ then for any $1\le \delta<1/\xi^+,$ $e^{\delta\varphi}\in A_\infty(Q)$ and
$$
[e^{\delta \varphi}]_{A_\infty(Q)}\le \frac{e^{-\delta\xi^+}}{1-\delta\xi^+}.
$$
This estimate is sharp.
\end{theorem}

\subsection{The $\BMO^p$-$A_\infty$ connection}

Let us describe the key idea behind Theorem~\ref{t1}. The known results for $p=1$ (\cite{korenovsky,lerner})  use techniques, such as equimeasurable rearrangements, that do not produce sharp constants for other $p.$ The results for $p=2$ 
(\cite{sv,vv}) rely on explicit Bellman functions for the corresponding inequalities. We also use this technique, but in a special, dual formulation. The following brief description of the Bellman approach to $\BMO^2$ will help explain the challenges and solutions for the problem at hand.

If one works on $\BMO^2,$ the expression for the norm in~\eqref{1} simplifies:
$$
\|\varphi\|_{\BMO^2(Q)}=\sup_{\text{cube}~J\subset Q}\big(\av{\varphi^2}J-\av{\varphi}J^2\big)^{1/2}.
$$ 
This allows for natural variational formulations. For example, to determine $\ve_0(2)$ and $C(\ve,2)$ in~\eqref{jn2}, the following extremal problem was stated and solved in~\cite{sv}:
\eq[0.1]{
\bel{B}_{\ve}(x_1,x_2)=\sup\{\av{e^\varphi}Q: ~\av{\varphi}Q=x_1,~\av{\varphi^2}Q=x_2,~\|\varphi\|_{\BMO^2(Q)}\le\ve\}.
}
From its definition, this function can be seen to be independent of $Q.$ It is then found as a locally concave solution of the homogeneous \ma PDE on the plane domain
$\{x\in\mathbb{R}^2: x_1^2\le x_2\le x_1^2+\ve^2\},$
which means that the domain is foliated by straight-line characteristics along which the function is affine.
With $\bel{B}_\ve$ in hand, one determines the critical value of $\ve$ for which this function seizes being finite; that is precisely $\ve_0(2).$ To prove sharpness, optimizing functions are then constructed along the \ma characteristics.

One can take this idea further and pose a general problem of two-sided integral estimates on $\BMO^2:$ find the upper Bellman function
\eq[0.2]{
\bel{B}_{f,\ve}(x_1,x_2)=\sup\{\av{f(\varphi)}Q:~\av{\varphi}Q=x_1,~\av{\varphi^2}Q=x_2,~\|\varphi\|_{\BMO^2(Q)}\le\ve\},
}
as well as its lower counterpart, where the supremum is replaced with infimum.
Article~\cite{sv1} lays the foundation of this general theory: it develops several canonical blocks out of which \ma foliations are assembled and applies them to the case $f(t)=|t|^p$ for $p>0.$ The case $f(t)=\chi_{(-\infty,-\lambda]\cup[\lambda,\infty)}(t)$ was considered in~\cite{vv}. In the large paper~\cite{iosvz2} (also see the report~\cite{iosvz1}), the theory of canonical foliations is completed, at least for sufficiently nice $f.$ 
The recent paper~\cite{sz} formalizes the following {\it a priori} connection between developable graphs on the domain $\{x_1^2\le x_2\le x_1^2+\ve^2\}$ and integral estimates for $\BMO^2,$ which has been shown true in all known applications:  the upper Bellman function~\eqref{0.2} is the smallest locally concave function $B$ such that $B(t,t^2)=f(t);$  the lower Bellman function is the largest locally convex function satisfying this condition. This theorem actually holds for general domains in $\mathbb{R}^2$ that are differences of two convex sets; this is illustrated in the next section.

However, when one considers $\BMO^p$ with $p\ne 2,$ which is what we need to answer our main questions, this approach no longer works. We can formally define the Bellman function for each $p$ by analogy with~\eqref{0.1}; here is one reasonable definition:
\eq[0.3]{
\bel{\mathcal{B}}_{p,\ve}(x_1,x_2)=\sup\{\av{e^\varphi}Q:~\av{\varphi}Q=x_1, \av{|\varphi-\av{\varphi}Q|^p}Q=x_2,~\|\varphi\|_{\scriptscriptstyle \BMO^p(Q)}\le\ve\}.
}
However, the dynamics of the problem is unclear: the variables do not split in a quantifiable fashion when the interval $Q$ is split (this is in contrast with the situation in~\eqref{0.1} where if $Q$ is a disjoint union of $Q^-$ and $Q^+,$ then the vector $x_Q$ corresponding to $Q$ is a convex combination of the vectors $x_{Q^-}$ and $x_{Q^+}$).
In particular, this function does not satisfy any apparent PDE. Simply put, one has no way of computing it directly.

We overcome this problem in an arguably intuitive way: rather than estimating $\ave{e^{\varphi-\ave{\varphi}}}$ through $\|\varphi\|_{\BMO^p}$ to determine the critical value of the norm, we study the dual problem of estimating, from below, $\BMO^p$ oscillations of logarithms of $A_\infty$ weights and computing their asymptotics as the $A_\infty$-characteristic goes to infinity. The following general result, which is valid in all dimensions and has Theorem~\ref{t1} as a partial corollary, makes this precise.

Fix a $p>0$ and $C\ge1.$ Let
\eq[domain]{
\Omega_C=\{x\in\mathbb{R}^2: e^{x_1}\le x_2\le C\,e^{x_1}\}.
}
For a cube $Q$ and every $x=(x_1,x_2)\in\Omega_C,$ let 
\eq[adm]{
E_{x,C,Q}=\{\varphi\in L^1(Q):~\av{\varphi}Q=x_1,~\av{e^\varphi}Q=x_2,~[e^\varphi]_{A_\infty(Q)}\le C\}
}
When $Q$ and $C$ are fixed, we call the elements of $E_{x,C,Q}$ test functions corresponding to $x.$ It is easy to show that $E_{x,C,Q}$ is non-empty for every $x\in\Omega_C;$ in Section~\ref{optimizers} this is done for the case when $Q$ is an interval. Consider the following lower Bellman function:
\eq[bh]{
\bel{b}_{p,C}(x)=\inf\{\av{|\varphi|^p}Q:~\varphi\in E_{x,C,Q}\}.
}

\begin{theorem}
\label{main}
Assume that there exists a family of functions $\{b_C\}_{C\ge1}$ such that for each~$C$ $b_C$ is defined on $\Omega_C,$ $b_C\le \bel{b}_{p,C},$ and $b_C(0,\cdot)$ is continuous on the interval $[1,C].$ Then
\eq[mt1]{
\ve^p_0(p)\ge \limsup_{C\to\infty} b_C(0,C).
}
Furthermore, let $G(C)=\big(b_C(0,C)\big)^{1/p}.$ If there exists $C_0\ge1$ such that $G$ is strictly increasing on the interval $[C_0,\infty),$ then the function $C(\ve,p)$ from~\eqref{jn2} satisfies
\eq[mt2]{
C(\ve,p)\le G^{-1}(\ve),\quad  \text{for}\quad G(C_0)\le \ve<\ve_0(p).
}
\end{theorem}

The thrust of this theorem is that to get good estimates for $\ve_0(p)$ and $C(\ve,p)$ one needs good lower estimates for the function $\bel{b}_{p,C}.$ We actually manage to find that function itself, which is the reason we obtain the exact $\ve_0(p).$ However, our findings are limited to the range $p\in[1,2]$ and dimension~1. As explained in Section~\ref{bellman}  below, the limitation on $p$ is of technical nature and can be overcome with additional effort. However, the restriction to the one-dimensional case can currently be avoided only at the cost of introducing exponential dependence on dimension, which would defeat the purpose of this convoluted setup. Let us point out that if one finds a non-trivial dimension-free lower estimate for $\bel{b}_{p,C}$ for any $p>0$ ($p=1$ is, perhaps, the easiest case) one will have proven a dimension-free John--Nirenberg inequality, a lofty goal.

Additionally, one can consider a function $h:[0,\infty)\to [0,\infty)$ and prove analogs of this theorem for ``$\BMO^h$'' whereby the $p$-based BMO norm is replaced with $\sup_J\av{h(|\varphi-\av{\varphi}J|)}J$. Here $h$ can be pretty wild -- it was recently shown in~\cite{weakBMO} that a single assumption, $\lim_{t\to\infty}h(t)=\infty,$ is enough to conclude that $\BMO^h\subset \BMO$ as sets. On the other hand, if $h$ is sufficiently regular, one can realize $\BMO^h$ as a Banach space by means of Orlicz norms, see~\cite{stromberg}.

Like Theorem~\ref{t1}, Theorem~\ref{t2} also uses $A_\infty$ estimates, but in a more direct manner. The rough idea is this: if $p>1$ and $e^\varphi\in A_\infty,$ then Theorem~\ref{t8} gives an upper estimate on the distribution function of $\varphi-\av{\varphi}J$ as a multiple of $e^{-\lambda/\xi^+},$ while Theorem~\ref{t3} gives the lower estimate $\|\varphi\|_{\BMO^p}\ge\ve_0(p)\xi^+$ for large enough $C.$ Comparing inequalities~\eqref{est1} and~\eqref{est2} makes clear why this does not work for $p=1:$ since $k(C)<2\xi^+/e$ for all $C,$ we cannot conclude that
$\|\varphi\|_{\BMO^1}\ge \ve_0(1)\,\xi^+$ holds for all $\varphi$ with $[e^\varphi]_{A_\infty}=C.$

Finding best constants for various BMO inequalities is a popular pursuit, motivated by the natural desire for sharpness and the central role BMO plays in modern analysis. In particular, such constants may have implications for $L^\infty\to\BMO$ norms of singular integrals (Corollary~\ref{2.8} is a simple illustration), and by interpolation, for their norms on $L^p;$ this is especially important in higher dimensions. Besides those already mentioned, we note papers~\cite{css}, \cite{knese}, \cite{ose1}, \cite{ose2}, \cite{ose3}, and \cite{sv2}. The last four use Bellman functions or their variants; the last two deal with the dyadic version of BMO.

When seeking sharp constants, one often wants to know whether they are attained and what the optimizing functions or sequences are. A key feature of the Bellman-function approach (at least of the precise version we use) is that the constants and the optimizers are arrived at together: an explicit formula for a Bellman function represents a continuum of sharp inequalities parametrized by points of its domain, while the function's differential structure -- in our application, the straight-line trajectories generated by the kernel of its Hessian -- gives a complete picture of optimizers through a now-standard procedure.

The rest of the paper is organized as follows: in Section~\ref{bellman}, we define all Bellman functions we use, state the main theorem that gives explicit formulas for these functions, and briefly discuss the method; in Section~\ref{nonB}, we assume the main Bellman theorem and prove all results stated above; Section~\ref{induction} contains the proof of one half of the main theorem, the so-called direct inequalities; finally, in Section~\ref{optimizers}, for each Bellman function we give explicit optimizers,
thus finishing the proof of the main theorem.

\section{The main Bellman theorem and discussion}
\label{bellman}

Fix $C\ge1$ and an interval $Q$ and recall the definitions \eqref{domain} of $\Omega_C$ and~\eqref{adm} of $E_{x,C,Q}.$ For $R> 0,$ let
$$
\Gamma_R=\{x\in\mathbb{R}^2:~x_2=Re^{x_1}\};
$$
then $\Omega_C$ is the region in the plane bounded below by $\Gamma_1$ and above by $\Gamma_C.$

We will need four Bellman functions on $\Omega_C,$ one lower and three upper.
For $p\in[1,2],$ $\delta\ge1,$ and $\lambda\in\mathbb{R},$ let
\eq[belp]{
\bel{b}_{p,C}(x)=\inf \{\av{|\varphi|^p}Q\!:~\varphi\in E_{x,C,Q}\},
}
\eq[Bel2]{
\bel{B}_{2,C}(x)=\sup \{\av{\varphi^2}Q\!:~\varphi\in E_{x,C,Q}\},
}
\eq[A]{
\bel{A}_{\delta,C}(x)=\sup \{\av{e^{\delta\varphi}}Q\!:~\varphi\in E_{x,C,Q}\},
}
\eq[D]{
\bel{D}_{\lambda,C}(x)=\sup \Big\{\frac1{|Q|}|\{t\in Q:~\varphi(t)\ge\lambda\}|:~\varphi\in E_{x,C,Q}\Big\}.
}
For each of these functions we now specify the Bellman candidate, i.e., an explicit function on $\Omega_C$ that will then be shown to equal the Bellman function. % We will denote each candidate by the same letter and with the same indices as the corresponding Bellman function, but without the bold. 

First, recall the functions~$\xi^\pm(C)$ from~\eqref{xi2}.
Geometrically, $\xi^\pm$ are the horizontal coordinates of the two points of tangency when two tangents to $\Gamma_C$ are drawn from the point $(0,1).$ In addition, let us define two new functions on $\Omega_C,$ $u^+$ and $u^-,$ by the implicit formula
\eq[upm]{
x_2=e^{u^\pm}\Big(\frac{x_1-u^\pm}{1-\xi^{\scriptscriptstyle \pm}}+1\Big).
}
To illustrate their geometric meaning, take a point $x\in\Omega_C$ and draw two one-sided tangents to $\Gamma_C,$ so that each tangent starts at $\Gamma_1,$ passes through $x,$ and terminates at the point of tangency. One of these tangents, call it $\ell^+(x),$ will have its point of tangency to the right of $x;$ the other, call it $\ell^-(x),$ has the point of tangency to the left of $x.$ In the first case, the horizontal coordinate of the initial point is $u^+(x)$ and that of the point of tangency is $u^+(x)+\xi^+;$ in the second case, these are $u^-(x)$ and $u^-(x)+\xi^-,$ respectively; see Figure~\ref{fig_u}. %Lastly, it is instructive to note  that $u^\pm$ can be written in the following form, though we will not use that here:
%$$
%u^\pm(x)=x_1+\xi^\pm\Big(\frac{Ce^{x_1}}{x_2}\Big)-\xi^\pm(C)
%$$
\newpage
\begin{figure}[!h!]
\includegraphics[width=10.5cm]{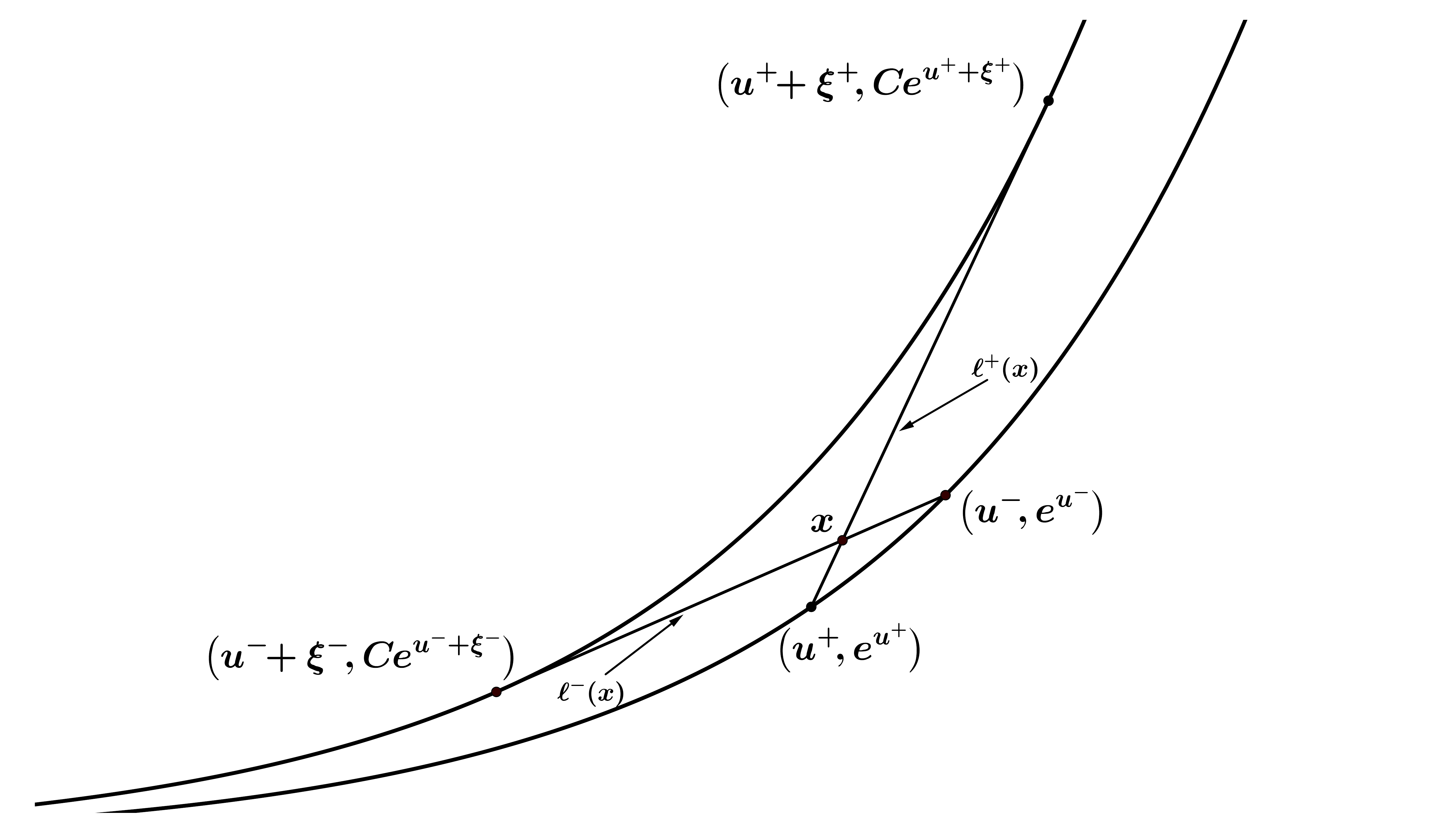}
\caption{The geometric meaning of $\xi^\pm,$ $u^\pm(x),$ and $\ell^\pm(x)$}
\label{fig_u}
\end{figure}

If $p\in(1,2],$ the candidate for $\bel{b}_{p,C}$ is given by 
\eq[bp]{
b_{p,C}(x)=\frac p{\xi^{\scriptscriptstyle+}}\,e^{u^{\scriptscriptstyle +}/\xi^{\scriptscriptstyle +}}
\Big[\int_{u^{\scriptscriptstyle +}}^\infty |s|^{p-1}\sgn(s)\,e^{-s/\xi^{\scriptscriptstyle +}}ds\Big](x_1-u^+)+|u^+|^p.
}
For future reference, let us specify this function for $p=2:$
\eq[b2]{
b_{2,C}(x)=2u^+x_1-(u^+)^2+2\xi^+(x_1-u^+).
}

The upper counterpart of the lower-Bellman candidate $b_{2,C}$ is
\eq[B2]{
B_{2,C}(x)=2u^-x_1-(u^-)^2+2\xi^-(x_1-u^-).
}

To define the candidate for $\bel{b}_{1,C},$ we need to split $\Omega_C$ into three subdomains separated by the tangents $\ell^\pm{(0,1)}$ and pictured in Figure~\ref{f2}. These subdomains may share boundaries, but are otherwise disjoint.  Thus, $\Omega_C=\Omega_-\cup\Omega_0\cup\Omega_+,$ where
\eq[b1domain]{
\begin{aligned}
&\Omega_-\!\!=\!\big\{\!\!-\infty< x_1\le \xi^-\!,~e^{x_1}\le x_2 \le Ce^{x_1}\big\}
\cup\big \{\xi^-\le x_1\le0,~e^{x_1}\le x_2 \le Ce^{\xi^-}\!\!x_1\!+\!1\big\}\\
&\Omega_0\!=\!\big\{\xi^-\!\le x_1\le0,~Ce^{\xi^-}\!\!x_1\!+\!1\le  x_2 \le Ce^{x_1} \big\}
\cup \big\{0\le x_1\le \xi^+\!,~Ce^{\xi^+}\!\!x_1\!+\!1\le x_2\le Ce^{x_1}\!\big\}\\
&\Omega_+\!\!=\!\big\{ 0\le x_1\le \xi^+\!,~e^{x_1}\le x_2 \le Ce^{\xi^+}\!\!x_1+1\big\}
\cup \big\{\xi^+\le x_1<\infty,~e^{x_1}\le x_2 \le Ce^{x_1}\big\}
\end{aligned}
}

\begin{figure}[h]
%\hspace{2.5cm}
\includegraphics[width=10.5cm]{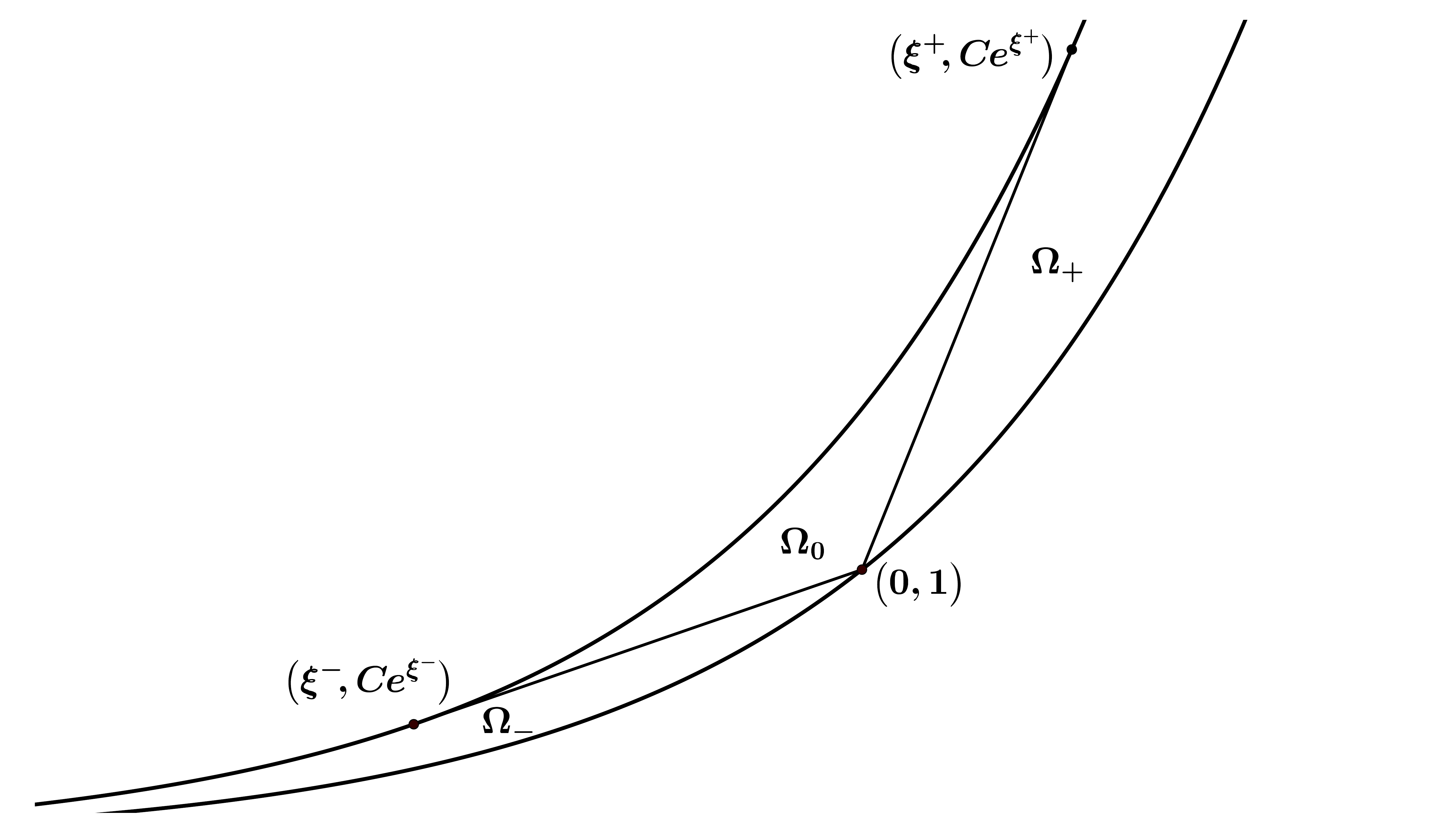}
\caption{The splitting of $\Omega_C$ for $\bel{b}_{1,C}$}
\label{f2}
\end{figure}
Then the candidate is given by
\eq[b1]{
b_{1,C}(x)=
\begin{cases}
-x_1,&x\in \Omega_-
\bigskip

\\
\ds2\,\frac{(1-\xi^-)(1-\xi^+)}{\xi^+-\xi^-}\,(x_2-1)+\frac{\xi^++\xi^--2}{\xi^+-\xi^-}\,x_1,&x\in \Omega_0
\bigskip

\\
~x_1,&x\in \Omega_+.
\end{cases}
}

The candidate for the function $\bel{A}_{\delta,C}$ is
\eq[Ad]{
A_{\delta,C}(x)=e^{\delta u^{\scriptscriptstyle+}}\left(\frac{\delta}{1-\delta \xi^+}\,(x_1-u^+)+1\right).
}

Finally, to give the candidate for the weak-type Bellman function~\eqref{D}, we need to split $\Omega_C$ into four sub-domains, $\Omega_C=\cup_{k=1}^4\Omega_k(\lambda).$ The splitting is illustrated in Figure~\ref{f3}.

\eq[wfd]{
\begin{aligned}
\Omega_1(\lambda)&=\big\{-\infty< x_1\le\lambda+\xi^--\xi^+,~e^{x_1}\le  x_2\le Ce^{x_1}\big\}\\
&\hspace{.5cm}\cup\big\{\lambda+\xi^--\xi^+\le x_1\le\lambda +\xi^-,~e^\lambda(Ce^{\xi^-}\!(x_1\!-\lambda)+1)\le x_2\le Ce^{x_1}\big\}
\\
\Omega_2(\lambda)&=\big\{\lambda+\xi^-\!\!-\xi^+\le x_1< \lambda,~e^{x_1}\le x_2\le e^\lambda(Ce^{\xi^-}\!(x_1-\lambda)+1)\big\}\\
\Omega_3(\lambda)&=\big\{\lambda+\xi^-\!\le x_1\le\lambda,~e^\lambda(Ce^{\xi^-}\!(x_1\!-\lambda)+1)\le x_2\le Ce^{x_1}\big\}\\
&\hspace{.5cm}\cup\big\{\lambda\le x_1\le \lambda+\xi^+,~e^\lambda(Ce^{\xi^+}\!(x_1\!-\lambda)+1)\le x_2\le Ce^{x_1}\big\}\\
\Omega_4(\lambda)&=\big\{\lambda\le x_1\le \lambda+\xi^+,~ e^{x_1}\le x_2\le e^\lambda(Ce^{\xi^+}\!(x_1\!-\lambda)+1)\big\}\\
&\hspace{.5cm}\cup \{\lambda+\xi^+\le x_1<\infty,~e^{x_1}\le x_2 \le Ce^{x_1}\}
\end{aligned}
}
\begin{figure}[h]
%\hspace{2.5cm}
\includegraphics[width=14cm]{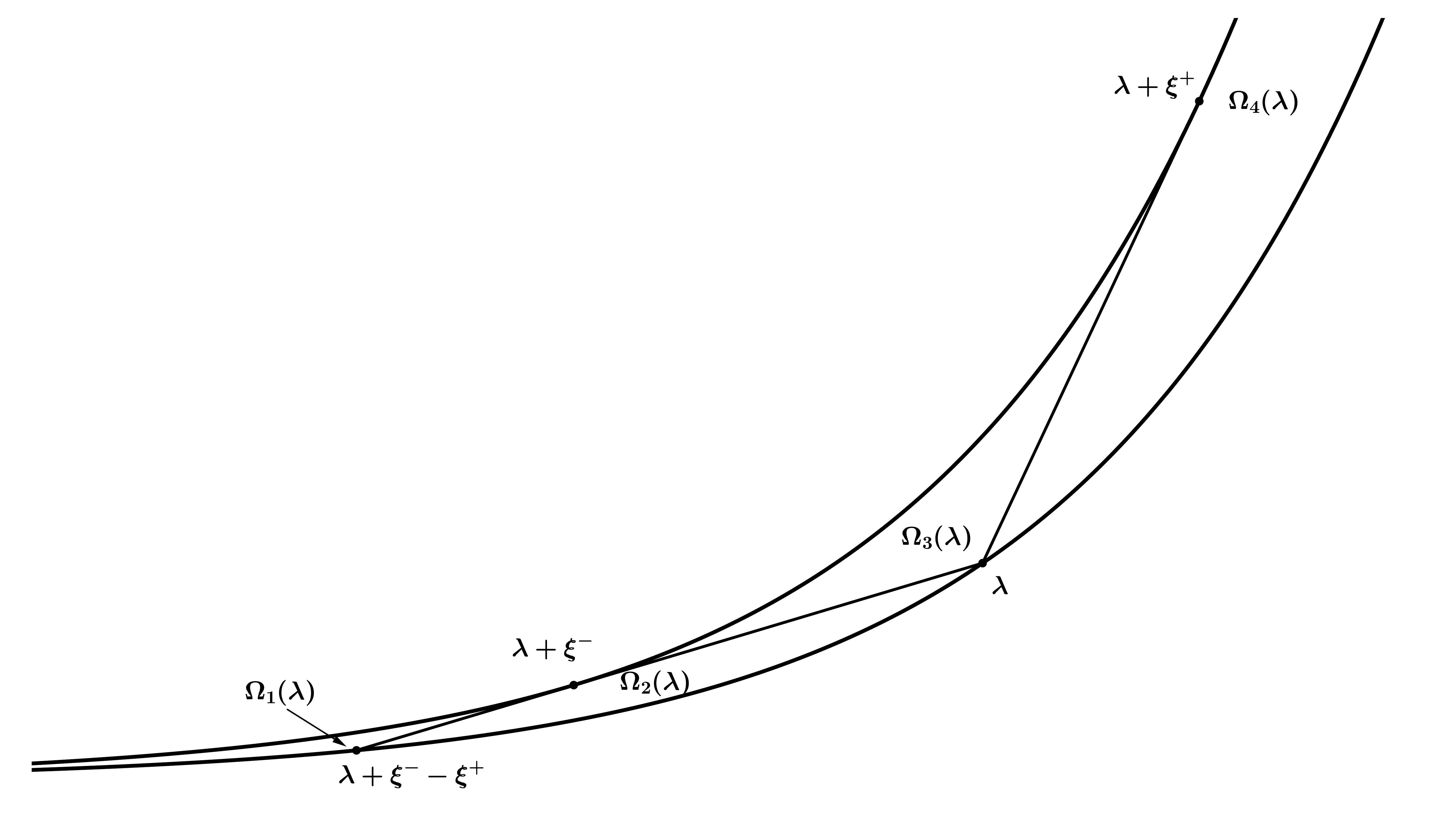}
\caption{The splitting of $\Omega_C$ for $\bel{D}_{\lambda,C}$}
\label{f3}
\end{figure}

\noindent We will also need the following auxiliary function in $\Omega_2(\lambda):$ Let $v=v(x)$ be the horizontal coordinate of the second point where the line through $(\lambda,e^{\lambda})$ and $x$ intersects $\Gamma_1.$ Thus,
\eq[v]{
\frac{e^v-e^\lambda}{v-\lambda}=\frac{x_2-e^\lambda}{x_1-\lambda}.
}
We are now in a position to define the Bellman candidate for $\bel{D}_{\lambda,C}.$ Let
\eq[wf]{
D_{\lambda,C}(x)=
\begin{cases}
\ds 1,& x\in\Omega_4(\lambda)
\bigskip

\\
\ds \frac{1-\xi^-}{(\xi^+-\xi^-)^2}\Big[x_1-\lambda+(1-x_2e^{-\lambda})(1-\xi^+)\Big]+1,&x\in\Omega_3(\lambda)
\bigskip

\\
\ds \frac{x_1-v}{\lambda-v},& x\in\Omega_2(\lambda)
\bigskip

\\
\ds \frac{e^{\frac{\xi^+-\xi^-+u^+-\lambda}{\xi^+}}}{\xi^+-\xi^-}\,(x_1-u^+),& x\in\Omega_1(\lambda).
\end{cases}
}
Note that $D_{\lambda,C}$ is continuous on $\Omega_C,$ except at the point $(\lambda,e^\lambda).$ By definition~\eqref{wfd}, that point belongs to both $\Omega_3(\lambda)$ and $\Omega_4(\lambda),$ but not to $\Omega_2(\lambda).$
  
Here is our main theorem of this section, which serves as the quantitative basis for all theorems stated in Section~\ref{main_results}.
\begin{theorem}
\label{all_bellman}
Let $C\ge 1.$
\ben[leftmargin=*]
\item[\bf (1)]
If $p\in(1,2]$ and $C\ge \frac{e^{p-2}}{p-1},$ then 
$
\bel{b}_{p,C}=b_{p,C}.
$
In particular, for any $C\ge1,$
$
\bel{b}_{2,C}=b_{2,C}.
$

\item[\bf (2)]
$
\bel{b}_{1,C}=b_{1,C}
$
\item[\bf (3)]
$
\bel{B}_{2,C}=B_{2,C}
$
\item[\bf (4)]
If $1\le \delta<1/\xi^+,$ then
$
\bel{A}_{\delta,C}=A_{\delta,C}.
$
If $\delta\ge 1/\xi^+,$ then
$$
\bel{A}_{\delta,C}(x)=
\begin{cases}
x_1,& \text{if}\quad x_2=e^{x_1};\\
\infty, & \text{if}\quad x_2>e^{x_1}.
\end{cases}
$$
\item[\bf (5)]
For any $\lambda\in\mathbb{R},$
$
\bel{D}_{\lambda,C}=D_{\lambda,C}.
$
\een

\end{theorem}
\subsection{Discussion}
We will not go into how our Bellman candidates were obtained, except to note that their constructions are close both in spirit and in technical details to the work on $\BMO^2$ presented in~\cite{sv1},~\cite{vv}, and~\cite{iosvz2}, and that the reader is encouraged to consult those papers. However, we would like to formulate a concise takeaway about the geometric nature of these candidates and outline the proof of Theorem~\ref{all_bellman} in the process. 

First, observe that if a function $\varphi$ on $Q$ is such that $\av{e^\varphi}Q=e^{\sav{\varphi}Q},$ then $\varphi$ is {\it a.e.} constant on $Q.$ Therefore the set $E_{(t,e^t),C,Q}$ contains only the constant function $\varphi(t)=t,$ and our Bellman functions automatically satisfy the following boundary conditions on $\Gamma_1:$
\eq[bcond]{
\bel{b}_{p,C}(t,e^t)=|t|^p,\quad\bel{B}_{2,C}(t,e^t)=t^2,\quad\bel{A}_{\delta,C}(t,e^t)=e^{\delta t},\quad \bel{D}_{\lambda,C}(t,e^t)=\chi_{_{[\lambda,\infty)}}(t).
}
As must be the case, and as can be easily checked, our candidates also satisfy these conditions. This fact is important in the proof of the theorem, which consists of two standard steps. 

Let us take $\bel{b}_{p,C}$ as an example. The first step, carried out in Section~\ref{induction}, is to show that
\eq[bb]{
b_{p,C}(x)\le \bel{b}_{p,C}(x),\quad x\in\Omega_C.
}
To that end, we first show that $b_{p,C}$ is locally convex in $\Omega_C$ (i.e., convex on every convex subset of $\Omega_C$). This local convexity, coupled with a geometric lemma due to Vasyunin, is then used in a special process called Bellman induction, eventually yielding a two-dimensional Jensen's inequality for the non-convex domain $\Omega_C:$
$$
b_{p,C}(\av{\varphi}Q,\av{e^\varphi}Q)\le \av{b_{p,C}(\varphi,e^\varphi)}Q=\av{|\varphi|^p}Q,
$$
where we used that $b_{p,C}(t,e^t)=|t|^p.$ Taking the infimum on the right over all $\varphi$ with fixed $\av{\varphi}Q$ and $\av{e^\varphi}Q$ gives~\eqref{bb}. This argument shows that $\bel{b}_{p,C}$ is bounded from below by any locally convex function with the right boundary conditions. 

The second step, done in Section~5, is to present for each $x\in\Omega_C$ an optimizer, i.e., a function $\varphi\in E_{x,C,Q}$ such that $\av{|\varphi|^p}Q=b_{p,C}(x).$ This automatically gives $\bel{b}_{p,C}(x)\le b_{p,C}(x),$ completing the proof.  Therefore, $b_{p,C}$ is the largest locally convex function on $\Omega_C$ satisfying the first condition in~\eqref{bcond}. Similarly, $B_{2,C},$ $A_{\delta,C},$ and $D_{\lambda,C}$ are the smallest locally concave functions on $\Omega_C$ satisfying their respective boundary conditions. This provides another illustration of the general theorem from~\cite{sz} mentioned in the previous section.

Consistent with our candidates' extremal nature, their graphs are so-called ruled surfaces, meaning that through each point on the graph passes a straight line contained in the graph. The traces of these lines in $\Omega_C$ are themselves straight lines, along which the candidate is affine; we refer to these traces as the \ma characteristics of the candidate. The reason for the name is that at points where the candidate -- call it $G$ -- is twice-differentiable, it satisfies the \ma equation $G_{x_1x_1}G_{x_2x_2}=G_{x_1x_2}^2,$ and the kernel of its Hessian generates the characteristics. They foliate the domain, and there is only one characteristic passing through each point, unless the candidate is affine in a neighborhood of that point; thus, the characteristics uniquely determine the candidate. Furthermore, the optimizers of Section~5 are built along these characteristics, using a procedure adapted from~\cite{sv1}. 

Let us briefly discuss the limitations of Theorem~\ref{all_bellman}. First, we restrict part (1) to $p\in[1,2]$ and $C\ge\frac{e^{p-2}}{p-1}$ simply because we do not yet know the \ma foliations, and thus the candidates, for other $p$ and other $C;$ they are certainly more complicated than the ones we have here. Once those are obtained, they will give lower estimates for $\ve_0(p)$ according to Theorem~\ref{main}, but it is not clear if they will give the exact value of $\ve_0(p)$ or allow us to prove an analog of Theorem~\ref{t2}; that will depend on the nature of their optimizers. Second, Vasyunin's lemma used to show~\eqref{bb} is specific to dimension~1, and its analogs in higher dimensions may well be false. The only currently available way to use Bellman induction on non-convex domains in higher dimensions is to replace the underlying function class with its dyadic, or similarly rigid, variant, which would give an exponential dependence on dimension (cf.~\cite{sv2}, where the John--Nirenberg constant for the dyadic $\BMO^2(\rn)$ is computed).

Lastly, the reader familiar with the results of~\cite{sv} may recognize $b_{2,C}$ as a special inverse of the upper Bellman function~$\bel{B}_\ve$ defined by~\eqref{0.1}, for which that paper provides an explicit formula. Without going into details, we note that
$$
\bel{B}_{\xi^+}(z_1,z_2)=e^{z_1}\big(b_{2,C}(0,\cdot)\big)^{-1}(z_2-z_1^2).
$$
A similar relation holds between $B_{2,C}$ and the lower counterpart of the function~$\bel{B}_\ve.$ More importantly, one can make similar statements for the Bellman functions $\bel{\mathcal{B}}_{\ve,p}$ defined by~\eqref{0.3}, thus representing those functions as ``shifted inverses'' of the functions $b_{p,C}$ that are the subject of the present work. This is a fairly wonderful fact, as $\bel{\mathcal{B}}_{\ve,p}$ do not admit direct computation.

\section{All non-Bellman proofs}
\label{nonB}
Here we first state and prove two auxiliary results and then, assuming Theorem~\ref{all_bellman}, prove all theorems and corollaries from Section~\ref{main_results} (except Corollary~\ref{2.8}) in this order: Theorem~\ref{main}, Theorem~\ref{t3}, Theorem~\ref{t8}, Theorem~\ref{t7}, Theorem~\ref{t1}, Theorem~\ref{t2}, and Corollary~\ref{t2.5}.

The following intuitive lemma is valid in all dimensions. Here $\BMO$ and $A_\infty$ are defined either on $\rn$ or on a cube $Q\subset\rn;$ in the latter case, all cubes are subcubes of $Q.$
 
\begin{lemma}
 \label{cont}
 Let $\varphi$ be a non-constant $\BMO$ function. For $\ve\in[0,\ve_\varphi),$ let $F(\ve)=[e^{\ve\varphi}]_{A_\infty}.$ Then $F$ is a strictly increasing, continuous function on $[0,\ve_\varphi),$ and $\lim_{\ve\to\ve_\varphi}F(\ve)=\infty.$
 \end{lemma}
 \begin{proof}
 Note that since $\varphi$ is not constant, $F(\ve)>0$ if $\ve>0.$ If $0\le\ve_1< \ve_2<\ve_\varphi,$ then by H\"older's inequality, for any cube $J,$
 $$
 \bav{e^{\ve_1(\varphi-\sav{\varphi}J)}}J\le  \bav{e^{\ve_2(\varphi-\sav{\varphi}J)}}J^{\ve_1/\ve_2}
 $$
 and so $F(\ve_1)\le F(\ve_2)^{\ve_1/\ve_2}< F(\ve_2),$  thus $F$ is strictly increasing.
 
Now, take any $\ve_*\in(0,\ve_\varphi),$ any $\delta\in(0,\ve_*),$ and any $\ve_1,\ve_2\in[0,\ve_*-\delta]$ such that $\ve_1\le\ve_2.$ For any cube $J$ we have
\begin{align*}
\bav{e^{\ve_2(\varphi-\sav{\varphi}J)}}J-\bav{e^{\ve_1(\varphi-\sav{\varphi}J)}}J&=
\bav{e^{\ve_2(\varphi-\sav{\varphi}J)}
\big[1-e^{(\ve_1-\ve_2)(\varphi-\sav{\varphi}J)}\big]}J\\
&\le(\ve_2-\ve_1)\bav{e^{\ve_2(\varphi-\sav{\varphi}J)}(\varphi-\av{\varphi}J)}J\\
&\le (\ve_2-\ve_1)\bav{e^{\ve_*(\varphi-\sav{\varphi}J)}}J^{\ve_2/\ve_*}
\bav{|\varphi-\av{\varphi}J|^{\ve_*/(\ve_*-\ve_2)}}J^{(\ve_*-\ve_2)/\ve_*}\\
&\le (\ve_2-\ve_1)\bav{e^{\ve_*(\varphi-\sav{\varphi}J)}}J
\bav{|\varphi-\av{\varphi}J|^{\ve_*/\delta}}J^{\delta/\ve_*}\\
&\le (\ve_2-\ve_1)F(\ve_*)\|\varphi\|_{\BMO^{\ve_*/\delta}}.
\end{align*}
Here we first used the elementary inequality $1-e^{t}\le -t,$ then H\"older's inequality, and then the monotonicity in $p$ of the expression $\av{|\varphi-\av{\varphi}J|^p}J^{1/p}.$

Therefore,
$$
 0\le F(\ve_2)-F(\ve_1)\le \sup_{\text{cube~J}} 
\big(\bav{e^{\ve_2(\varphi-\sav{\varphi}J)}}J-\bav{e^{\ve_1(\varphi-\sav{\varphi}J)}}J\big)
\le (\ve_2-\ve_1)F(\ve_*)\|\varphi\|_{\BMO^{\ve_*/\delta}},
$$
which means that $F$ is continuous on $[0,\ve_*-\delta]$ and so on $[0,\ve_\varphi).$

Lastly, let $L=\lim_{\ve\to\ve_\varphi}F(\ve)$ and assume that $L$ is finite. If $\ve_\varphi=\infty,$ then for any $\ve>0,$ $\tau>0,$ and cube $J,$
$$
e^{\ve\tau}\frac1{|J|}\big|\{t\in J:~\varphi-\av{\varphi}J>\tau\}\big|\le\bav{e^{\ve(\varphi-\sav{\varphi}J)}}J\le L.
$$
Taking the limit as $\ve\to\infty,$ we conclude that $|\{t\in J:~\varphi-\av{\varphi}J>\tau\}|=0.$ Since this is true for all $\tau>0,$ we have $\varphi\le\av{\varphi}J$ {\it a.e.} on $J,$ and, thus, $\varphi=\av{\varphi}J$ {\it a.e.} on $J.$ Since this holds for all $J,$ $\varphi$ is constant, a contradiction. 

If $\ve_\varphi<\infty,$ then by Fatou's lemma the expressions $\av{e^{\ve_\varphi(\varphi-\sav{\varphi}J)}}J$ are bounded uniformly in $J$ and so $e^{\ve_\varphi\varphi}$ is an $A_\infty$ weight. However, by Gehring's theorem there then exists $\eta>0$ such that 
$e^{\ve_\varphi(1+\eta)\varphi}\in A_\infty,$ which contradicts the definition of $\ve_\varphi.$
 \end{proof}
Most of our sharp inequalities are extremized by the same function, the logarithm. The following  lemma makes this precise.
For $p>0$ let
$$
\omega(p)=\left[\frac pe\Big(\Gamma(p)-\int_0^1t^{p-1}e^t\,dt\Big)+1\right]^{1/p}.
$$
\begin{lemma}
\label{phi0}
Let $\varphi_0(t)=\log(1/t),~t\in(0,1).$ Then
\eq[eps_phi]{
\ve_{\varphi_0}=1,\qquad \ve_{-\varphi_0}=\infty.
}
If $p\ge1,$ then
\eq[norm]{
\|\varphi_0\|_{\BMO^p((0,1))}=\omega(p).
}
Consequently,
\eq[norm1]{
\ve_0(p)\le \omega(p)
}
and
\eq[norm2]{
C(\ve,p)\ge \frac{e^{-\ve/\omega(p)}}{1-\ve/\omega(p)},\quad 0\le\ve<\omega(p).
}

\end{lemma}
\begin{remark}
In this paper, we only need~\eqref{norm}-\eqref{norm2} for $p\in[1,2],$ but the proof given below works for all $p\ge1.$ What is more, these statements actually hold for all $p>0,$ but showing this requires more involved computation.
\end{remark}
\begin{proof}
Let $Q=(0,1).$ To show~\eqref{eps_phi}, take any subinterval $I=(a,b)$ of $Q$ with $a>0.$ We have $\av{\varphi_0}I=\frac1{b-a}(a\log a-b\log b)+1$ and a simple calculation shows that for any $\ve\in(-\infty,1),$
$$
\bav{e^{\ve(\varphi_0-\sav{\varphi_0}I)}}I=\frac{e^{-\ve}}{1-\ve}\left[\theta^{-\frac{\ve\theta}{1-\theta}}\,\frac{1-\theta^{1-\ve}}{1-\theta}\right],
$$
where $\theta:=a/b.$ The expression in brackets is a decreasing function of $\theta$ for any $\ve\in(-\infty,1),$ and its limit as $\theta\to0$ (which corresponds to the case $a=0$) is 1. Therefore,
\eq[char]{
[e^{\ve\varphi_0}]_{A_\infty}=\frac{e^{-\ve}}{1-\ve},\quad 0\le\ve<1,\qquad\text{and}\qquad [e^{-\ve\varphi_0}]_{A_\infty}=\frac{e^{\ve}}{1+\ve},\quad0\le\ve<\infty,
}
which proves~\eqref{eps_phi}.

To show~\eqref{norm}, first note that if $b>0,$ then $\av{\varphi_0}{(0,b)}=1-\log b$ and
\begin{align*}
\av{|\varphi_0-\av{\varphi_0}{(0,b)}|^p}{(0,b)}&=\frac1b\,\int_0^b|-\log t+\log b-1|^p\,dt
=\int_0^1|\log t+1|^p\,dt\\
&=\int_{-\infty}^1|s|^pe^{s-1}\,ds=\frac pe\Big(\Gamma(p)-\int_0^1t^{p-1}e^t\,dt\Big)+1=\omega^p(p).
\end{align*}
Thus, it remains to show that $\av{|\varphi_0-\av{\varphi_0}I|^p}I\le\int_0^1|\log t+1|^p\,dt$ when $I=(a,b)\subset Q$ with $a>0.$
After the change of variable $t\mapsto t\,a^{a/(b-a)}b^{-b/(b-a)}$ we have
\begin{align*}
\av{|\varphi_0-\av{\varphi_0}I|^p}I&=\frac1{b-a}\,\int_a^b|-\log t-\av{\varphi_0}I|^p\,dt=\frac1{z-w}\,\int_{w}^{z}f(t)\,dt=:U,
\end{align*}
where we set $f(t)=|\log t+1|^p,$ $w=(a/b)^{b/(b-a)},$ and $z=(a/b)^{a/(b-a)}.$ It is easy to check that $0<w<1/e<z<1$ and that $w$ is given as a function of $z$ by the equation $w\log w=z\log z;$ that, in turn, defines $U$ as a function of $z.$
We would like to show that $U$ is increasing on the interval $(1/e,1)$ and, thus, $U(z)\le U(1)=\omega^p(p).$ To that end, we compute:
\begin{align*}
(z-w)U'(z)&=-U(z)(1-w')+f(z)-f(w)w'\\
&=\frac1{1+\log w}\,\left(\log(z/w)U(z)+f(z)(1+\log w)-f(w)(1+\log z)\right)\\
&=\frac1{1+\log w}\,\left[\log(z/w)U(z)+(1+\log w)(1+\log z)\big(f(w)^{1-1/p}+f(z)^{1-1/p}\big)\right],
\end{align*}
where we used that $w'=(1+\log z)/(1+\log w).$ Since $1+\log w<0,$ we need to show that the expression in brackets is non-positive. Since $f$ is decreasing on $(0,1/e)$ and increasing on $(1/e,1),$ and because $p\ge1,$ we can estimate $U$ as follows:
\begin{align*}
(z-w)U(z)&=\int_w^{1/e}f(t)\,dt+\int^z_{1/e}f(t)\,dt\\
&\le -f(w)^{1-1/p}\int_w^{1/e}(1+\log t)\,dt+f(z)^{1-1/p}\int_{1/e}^z(1+\log t)\,dt\\
&=(1/e+z\log z)\left(f(w)^{1-1/p}+f(z)^{1-1/p}\right).
\end{align*}
Therefore, it suffices to show that
$$
\log(z/w)(1/e+z\log z)\le -(1+\log w)(1+\log z)(z-w).
$$
Since $w\log w=z\log z,$ after simplification this becomes 
$$
1+\log z+\log w\Big(1-\frac1{ez}\Big)\le0.
$$
We leave it to the reader to verify that $\log w\le-ez.$ With this in mind, and because $z\ge 1/e,$ it suffices to show that
$2+\log z-ez\le0,$ which is elementary.
This proves~\eqref{norm}.

To show~\eqref{norm1}, note that $\ve_{\varphi_0}=1$ and so
$$
\ve_0(p)\le\inf\{\ve_\varphi:~ \|\varphi\|_{\BMO^p(Q)}=1 \}\le 
\ve_{\varphi_0/\|\varphi_0\|_{\scriptscriptstyle \BMO^p(Q)}}=\|\varphi_0\|_{\BMO^p(Q)}.
$$

Finally, \eqref{norm2} holds because if $\ve\in[0,\omega(p)),$ then
$$
C(\ve,p)\ge \bav{e^{\ve(\varphi_0-\sav{\varphi_0}Q)/\|\varphi_0\|_{\BMO^p(Q)}}}{Q}=e^{-\ve/\omega(p)}\int_0^1t^{-\ve/\omega(p)}\,dt
= \frac{e^{-\ve/\omega(p)}}{1-\ve/\omega(p)}.
\eqno\qedhere
$$
\end{proof}

\begin{proof}[Proof of Theorem~\ref{main}]
Take any $\varphi\in\BMO(Q)$ with $\|\varphi\|_{\BMO^p(Q)}=1.$ For all $\ve\in[0,\ve_\varphi),$ let
$
F_\varphi(\ve)=[e^{\ve\varphi}]_{A_\infty(Q)}.
$
By the definition of $\bel{b}_{p,C}$ and the assumption on $b_C,$ for any cube $J\subset Q,$
$$
b_{F_\varphi(\ve)}\big(\av{\ve\varphi}J,\av{e^{\ve\varphi}}J\big)\le\bel{b}_{p,F_\varphi(\ve)}\big(\av{\ve\varphi}J,\av{e^{\ve\varphi}}J\big)\le\av{|\ve\varphi|^p}J. 
$$
Replacing $\varphi$ with $\varphi-\av{\varphi}J$ gives
$$
b_{F_\varphi(\ve)}\big(0,\av{e^{\ve(\varphi-\av{\varphi}{\scriptscriptstyle J})}}J\big)\le \ve^p\av{|\varphi-\av{\varphi}J|^p}J 
\le\ve^p\|\varphi\|^p_{\BMO^p(Q)}=\ve^p.
$$
Take a sequence $\{J_n\}$ of subcubes of $Q$ such that
$$
\lim_{n\to\infty}\av{e^{\ve(\varphi-\av{\varphi}{\scriptscriptstyle J_n})}}{J_n}=F_\varphi(\ve).
$$
Since $b_C(0,\cdot)$ is continuous, we have
\eq[mt3]{
b_{F_\varphi(\ve)}(0,F_\varphi(\ve))\le\ve^p.
}
Recall that $\ve_0(p)\le \inf\{\ve_\varphi: \|\varphi\|_{\BMO^p}=1\}.$
Assume first that there exists $\varphi$ such that $\|\varphi\|_{\BMO^p}=1$ and $\ve_\varphi=\ve_0(p).$ For that particular $\varphi,$ take the limit as $\ve\to\ve_0(p)$ in~\eqref{mt3}. By Lemma~\ref{cont}, $F$ is continuous on $[0,\ve_0(p))$ and $F_\varphi(\ve)\to\infty$ as $\ve\to\ve_0(p).$ This gives~\eqref{mt1}.

If no such $\varphi$ exists, then for all $\varphi$ with $\|\varphi\|_{\BMO^p}=1$ we have $e^{\ve_0(p)\varphi}\in A_\infty,$ and by the definition of $\ve_0(p)$ there exists a sequence $\{\varphi_k\}$ of such functions for which 
$C_k:=F_{\varphi_k}(\ve_0(p))\to\infty$ as $k\to\infty.$ Using~\eqref{mt3} with $\varphi=\varphi_k$ and $\ve=\ve_0(p)$ we have
$$
b_{C_k}(0,C_k)\le\ve^p_0(p),
$$
and taking the limit as $k\to\infty$ again gives~\eqref{mt1}.

To show~\eqref{mt2}, observe that~\eqref{mt3} means that
$$
G(F_\varphi(\ve))\le \ve.
$$
By Lemma~\ref{cont}, there exists $\ve^*\in[0,\ve_\varphi)$ such that $F_\varphi(\ve^*)=C_0.$ Since $G$ is strictly increasing on $[C_0,\infty),$ we have
$$
F_\varphi(\ve)\le G^{-1}(\ve),\quad \ve^*\le\ve< \ve_\varphi,
$$
and so
$$
C(\ve,p)\le G^{-1}(\ve),\quad \ve^*\le \ve<\ve_0(p). \eqno\qedhere
$$
\end{proof}

\begin{proof}[Proof of Theorem~\ref{t3}]
Note that for all $p$ and for any interval $J\subset Q,$
\eq[ineq]{
\bel{b}_{p,C}\big(0,\av{e^{\varphi-\av{\varphi}{\scriptscriptstyle J}}}J\big)\le \av{|\varphi -\av{\varphi}J|^p}J
}
and
$$
 \av{\varphi^2}J- \av{\varphi}J^2\le \bel{B}_{2,C}\big(0,\av{e^{\varphi-\av{\varphi}{\scriptscriptstyle J}}}J\big),
$$
and both inequalities are sharp by the definitions of $\bel{b}_{p,C}$ and $\bel{B}_{2,C},$ in that for each inequality there exists a sequence $\{\varphi_n\}$ of functions on $J,$ such that the inequality becomes equality in the limit (in fact, as shown in Section~\ref{optimizers}, these sequences are realized as single functions).

The function $\bel{B}_{2,C}$ is given by~Theorem~\ref{all_bellman} together with~\eqref{B2}. We have $\bel{B}_{2,C}(0,x_2)=-(u^-(0,x_2))^2-2\xi^-u^-(0,x_2),$ and this is easily seen to be increasing in $x_2$ on the interval $[1,C].$ Since $u^-(0,C)=-\xi^-,$ 
$$
\av{\varphi^2}J-\av{\varphi}J^2
\le \bel{B}_{2,C}(0,C)=(\xi^-)^2.
$$
Note that this estimate is still sharp, and, hence, so is the one obtained by taking the supremum over all $J$ on the left, i.e., the upper inequality in~\eqref{est1}. 

To show the lower estimate for all $p$ take a sequence $\{J_n\}$ such that $\av{e^{\varphi-\av{\varphi}{\scriptscriptstyle J_n}}}{J_n}\to C.$ By Theorem~\ref{all_bellman}, $\bel{b}_{p,C}$ for $p\in(1,2]$ is given by~\eqref{bp} and $\bel{b}_{1,C}$ is given by~\eqref{b1}. In both cases $\bel{b}_{p,C}(0,x_2)$ is continuous in $x_2$ on the interval $[1,C].$  Thus,
$$
\bel{b}_{p,C}(0,C)\le\limsup_{n\to\infty} (\av{|\varphi -\av{\varphi}{J_n}|^p}{J_n})\le \|\varphi\|^p_{\BMO^p(Q)}.
$$
From~\eqref{bp}, $\bel{b}_{p,C}(0,C)=(\xi^+)^p\ve^p_0(p),$ and from~\eqref{b1} and~\eqref{kc}, $\bel{b}_{1,C}(0,C)=k(C).$ This proves the lower estimate in~\eqref{est1}, as well as both~\eqref{est2} and~\eqref{est3}. 

We now need to show sharpness in the first and third cases; note that it does not follow from the fact that the left-hand inequality in~\eqref{ineq} is sharp. Let $\varphi^+=\xi^+\log(1/t).$ By Lemma~\ref{phi0}, $\|\varphi^+\|_{\BMO^p((0,1))}=\xi^+\ve_0(p).$ On the other hand, by Lemma~\ref{phi0}, $e^{\varphi^+}\in A_\infty((0,1))$ and $[e^{\varphi^+}]_{A_\infty((0,1))}=e^{-\xi^+}/(1-\xi^+)=C.$ This completes the proof.
\end{proof}

\begin{proof}[Proof of Theorem~\ref{t8}]
Recall the Bellman function $\bel{D}_{\lambda,C}$ given by Theorem~\ref{all_bellman} and the formulas~\eqref{wfd},~\eqref{wf}; write $D$ for $\bel{D}_{\lambda,C}.$ Take $\varphi$ such that $e^\varphi\in \AC.$ Let $y=\av{e^{\varphi-\sav{\varphi}Q}}Q.$ By definition of $D,$ 
$$
\frac1{|Q|}\,\big|\{t\in Q:~\varphi(t)-\av{\varphi}Q\ge\lambda\}\big|\le D(0,y),
$$
and this estimate is sharp for any $y\in[1,C].$ Clearly, the same is true with any subinterval $J$ of $Q$ in place of $Q.$

We would like to sharply bound the right-hand side. That bound depends on the location of the point $(0,y)$ within $\Omega_C.$ Note that $D(0,y)$ is identically 1 if $(0,y)\in\Omega_4(\lambda);$ decreasing in $y$ if $(0,y)\in\Omega_3(\lambda);$ and increasing in $y$ if 
$(0,y)\in\Omega_1(\lambda)\cup\Omega_2(\lambda).$ (The last fact can be seen by means of direct differentiation, carried out in part (5) of the proof of Lemma~\ref{concon} in Section~\ref{induction}.) This means that if $0\ge \lambda,$ the best bound on $D(0,y)$ is trivial, i.e., $\max_yD(0,y)=1;$ if $\lambda+\xi^-\le 0\le \lambda,$ then $\max_yD(0,y)=D(0,e^{\lambda}(-Ce^{\xi^-}\lambda+1));$ and if $0\le \lambda+\xi^-,$ then $\max_yD(0,y)=D(0,C).$ After computing the last two quantities, we have 
$$
\frac1{|J|}\,\big|\{t\in J:~\varphi(t)-\av{\varphi}J\ge\lambda\}\big|\le\max_yD(0,y)=
\begin{cases}
1,& \lambda\le 0,
\bigskip

\\
1-\frac{\lambda}{\xi^+-\xi^-},& 0\le \lambda \le -\xi^-,
\bigskip

\\
\frac{\xi^+e^{-\xi^-/\xi^+}}{\xi^+-\xi^-}\,e^{-\lambda/\xi^+},& \lambda\ge -\xi^-,
\end{cases}
$$
and this estimate remains sharp.
\end{proof}

\begin{remark}
It will be useful to have an estimate for the right-hand side of~\eqref{ones} that involves only $\xi^+.$ It is a straightforward exercise to verify that $(1-\xi^-)^{1/\xi^+}\le \frac e2\,\big(1-\frac{\xi^-}{\xi^+} \big).$ Therefore,
\eq[simple]{
\frac{e^{-\xi^-/\xi^+}}{1-\xi^-/\xi^+}\le \frac{e}2\,\Big(\frac{e^{-\xi^-}}{1-\xi^-}\Big)^{1/\xi^+}=
\frac{e}2\,\Big(\frac{e^{-\xi^+}}{1-\xi^+}\Big)^{1/\xi^+}=\frac 12\,(1-\xi^+)^{-1/\xi^+}.
}
\end{remark}

\begin{proof}[Proof of Theorem~\ref{t7}]
By the definition of the Bellman function $\bel{A}_{\delta,C},$ for any $J\subset Q$
$$
\bav{e^{\delta(\varphi-\sav{\varphi}J)}}J\le \bel{A}_{\delta,C}(0,\av{e^{\varphi-\sav{\varphi}J}}J).
$$
The function $\bel{A}_{\delta,C}$ is given by Theorem~\ref{all_bellman} and formula~\eqref{Ad}. Observe that
$\bel{A}_{\delta,C}(0,\cdot)$ is increasing on $[1,C].$ We have $u^+(0,C)=-\xi^+,$ thus
$$
\bav{e^{\delta(\varphi-\sav{\varphi}J)}}J\le \bel{A}_{\delta,C}(0,C)=\frac{e^{-\delta\xi^+}}{1-\delta\xi^+}.
$$
This inequality is sharp by the definition of $\bel{A}_{\delta,C}.$ Taking supremum over all $J$ on the left, we obtain the sharp inequality in the statement of the theorem.
\end{proof}

\begin{proof}[Proof of Theorem~\ref{t1}]
In light of Lemma~\ref{phi0}, we only need to show the converse inequality to~\eqref{norm1} for all $p\in[1,2];$ the converse inequality to~\eqref{norm2} for $(2-p)\ve_0(p)\le \ve<\ve_0(p);$ and the upper inequality in~\eqref{ce1}. To that end,
we use Theorem~\ref{main} with $b_C=\bel{b}_{p,C}$ given by Theorem~\ref{all_bellman}. 

If $p>1,$ then $\bel{b}_{p,C}$ is given by~\eqref{bp}. From~\eqref{upm}, $u^+(0,C)=-\xi^+,$ so
\begin{align*}
\bel{b}_{p,C}(0,C)&=\frac p{e}\,
\Big[\int_{-\xi{\scriptscriptstyle +}}^\infty |s|^{p-1}\sgn(s)\,e^{-s/\xi^{\scriptscriptstyle +}}ds\Big]+|\xi^+|^p=(\xi^+)^p\left[\frac pe\left(\Gamma(p)-\int_0^1 t^{p-1}e^t\,dt\right)+1\right],
\end{align*}
where the second equality follows from splitting the integral into two and changing the variable. Since $\xi^+\to1$ as $C\to\infty,$ we have
$$
\ve_0(p)\ge\lim_{C\to\infty} \big(\bel{b}_{p,C}(0,C)\big)^{1/p}=\omega(p).
$$
Furthermore, the function $G(C):=(\bel{b}_{p,C}(0,C))^{1/p}=\xi^+(C)\ve_0(p)$ is strictly increasing on the interval $[\frac{e^{p-2}}{p-1},\infty)$ and $G\big(\frac{e^{p-2}}{p-1}\big)=(2-p)\ve_0(p),$ so~\eqref{mt2} gives
$$
C(\ve,p)\le G^{-1}(\ve)=\frac{e^{-\ve/\ve_0(p)}}{1-\ve/\ve_0(p)},\quad (2-p)\ve_0(p)\le \ve<\ve_0(p).
$$

If $p=1,$ then $\bel{b}_{p,C}$ is given by~\eqref{b1}, whence
$$
\bel{b}_{1,C}(0,C)=2\,\frac{(1-\xi^-)(1-\xi^+)}{\xi^+-\xi^-}\,(C-1)=k(C).
$$ 
and%Since $(1-\xi^-)/(\xi^+-\xi^-)\to 1$ and $(1-\xi^+)C=e^{-\xi^+}\to\frac1e$ as $C\to\infty,$ we have
$$
\ve_0(1)\ge\lim_{C\to\infty}k(C)=\frac2e.
$$
Lastly, since the function $k$ is strictly increasing on the interval $[1,\infty),$ we have
$$
C(\ve,1)\le k^{-1}(\ve).
$$
This completes the proof.
\end{proof}

\begin{proof}[Proof of Theorem~\ref{t2}]
Take $\varphi\in\BMO(Q)$ with $\|\varphi\|_{\BMO^p}\ne0$ and for all $0\le \ve<\ve_\varphi$  let $F(\ve)=[e^{\ve\varphi}]_{A^\infty(Q)}.$ Then $F(0)=1$ and by Lemma~\ref{cont}, $F(\ve)\to\infty$ as $\ve\to \ve_\varphi$ and $F$ is continuous on $[0,\ve_\varphi).$ Therefore, there exists an $\ve^*>0$ such that $F(\ve^*)=\frac{e^{p-2}}{p-1}.$ By Theorem~\ref{t3},
\eq[wt6]{
2-p=\xi^+\big(F(\ve^*)\big) \le \frac{\ve^*}{\ve_0(p)}\|\varphi\|_{\BMO^p}.
}
Using~\eqref{ones},~\eqref{simple}, and~\eqref{wt6}, we have
\begin{align*}
\frac1{|Q|}\,\big|\{t\in Q:~\ve^*(\varphi(t)-\av{\varphi}J)\ge\lambda\}\big|&\le \frac 12\,\left((1-\xi^+\big(F(\ve^*)\big)\right)^{-1/\xi^+(F(\ve^*))} \,e^{-\frac\lambda{\xi^+(F(\ve^*))}}\\
&\le\frac 12\,(p-1)^{-\frac1{2-p}} \,e^{-\frac{\ve_0(p)\lambda}{\ve^*\|\varphi\|_{\BMO^p}}}
\end{align*}

Replacing $\lambda/\ve^*$ with $\lambda$ gives
\eq[wt61]{
\frac1{|Q|}\,\big|\{t\in Q:~\varphi(t)-\av{\varphi}J\ge\lambda\}\big|\le \frac 12\,(p-1)^{-\frac1{2-p}}  \,e^{-\frac{\ve_0(p)\lambda}{\|\varphi\|_{\BMO^p}}}.
}
The same estimate holds with $\varphi$ replaced with $-\varphi,$ which means that we can obtain the inequality in the statement of the theorem by doubling the constant in~\eqref{wt61}.
\end{proof}

\begin{proof}[Proof of Corollary~\ref{t2.5}]
Take any $\varphi\in\BMO(Q).$ Without loss of generality, assume $\ve_\varphi<\infty.$ For $\ve\in[0,\ve_\varphi),$ let $F(\ve)=[e^{\ve\varphi}]_{A_\infty(Q)}.$ If $p\in(1,2]$ then by Lemma~\ref{cont}, for sufficiently large $\ve$ we have $F(\ve)\ge\frac{e^{p-2}}{p-1}$ and, thus, Theorem~\ref{t3} applies:
$$
\|\ve\varphi\|_{\BMO^p(Q)}\ge \xi^+(F(\ve))\ve_0(p).
$$
Take the limit as $\ve\to\ve_\varphi.$ Since $F(\ve)\to \infty,$ we have $\xi^+(F(\ve))\to1$ and so
\eq[dist3]{
\|\varphi\|_{\BMO^p(Q)}\ge\frac{\ve_0(p)}{\ve_\varphi}.
}
The same estimate holds with $-\varphi$ in place of $\varphi,$ thus,
\eq[dist4]{
\|\varphi\|_{\BMO^p(Q)}\ge\frac{\ve_0(p)}{\min\{\ve_{\varphi},\ve_{-\varphi}\}}.
}
Now, take any $f\in L^\infty(Q).$ Then $\ve_{\varphi-f}=\ve_\varphi$ and $\ve_{-\varphi+f}=\ve_{-\varphi},$ so
$$
\|\varphi-f\|_{\BMO^p(Q)}\ge\frac{\ve_0(p)}{\min\{\ve_{\varphi},\ve_{-\varphi}\}}.
$$
Taking the infimum over all such $f$ gives~\eqref{dist1}.

If $p=1,$ we have for all $\ve\in[0,\ve_\varphi),$
$$
\|\ve\varphi\|_{\BMO^1(Q)}\ge k(F(\ve)).
$$
Since $k(C)\to \frac2e=\ve_0(1)$ as $C\to\infty,$ taking the limit as $\ve\to\ve_\varphi$ again gives~\eqref{dist4} and, consequently,~\eqref{dist1}.  

To prove sharpness, let $Q=(0,1)$ and $\varphi_0(t)=\log(1/t)$ for $t\in Q.$ By Lemma~\ref{phi0}, $\|\varphi\|_{\BMO^p(Q)}=\ve_0(p).$ Since $\ve_{\varphi_0}=1$ and $\ve_{-\varphi_0}=\infty,$
$$
\inf_{f\in L^\infty(Q)}\|\varphi_0-f\|_{\BMO^p(Q)}\le \|\varphi_0\|_{\BMO^p(Q)}=\frac{\ve_0(p)}{\min\{\ve_{\varphi_0},\ve_{-\varphi_0}\}}.
$$

Finally, to show~\eqref{dist2}, take $\varphi\in\BMO(\mathbb{R}).$ For any interval $Q,$ we have $\varphi|_Q\in\BMO(Q);$ let $\ve_{\varphi,Q}=\ve_{\varphi|_Q}.$ Clearly, $\|\varphi\|_{\BMO(\mathbb{R})}\ge\|\varphi|_Q\|_{\BMO(Q)}$ and $\ve_{\varphi}=\inf_Q \ve_{\varphi,Q}.$ Therefore, by~\eqref{dist3},
$
\|\varphi\|_{\BMO(\mathbb{R})}\ge \frac{\ve_0(p)}{\ve_{\varphi,Q}}
$
for any $Q.$
Taking infimum over all $Q$ in the denominator on the right and combining the resulting estimate with the one for $-\varphi$ gives
$$
\|\varphi\|_{\BMO(\mathbb{R})}\ge \frac{\ve_0(p)}{\min\{\ve_{\varphi},\ve_{-\varphi}\}}.
$$ 
As before, for any $f\in L^\infty(\mathbb{R})$ we have $\ve_{\pm\varphi}=\ve_{\pm(\varphi-f)},$ and so~\eqref{dist2}  follows.
\end{proof}

\section{Bellman induction and direct inequalities}
\label{induction}

The main result of this section is the following ``half'' of Theorem~\ref{all_bellman}. 
\begin{lemma}
\label{lemma_induction}
Let $C\ge1.$
\ben[leftmargin=*]
\item[\bf (1)]
If $p\in(1,2]$ and $C\ge \frac{e^{p-2}}{p-1},$ then 
$
b_{p,C}\le \bel{b}_{p,C}.
$
\item[\bf (2)]
$
b_{1,C}\le \bel{b}_{1,C}
$
\item[\bf (3)]
$
B_{2,C}\ge \bel{B}_{2,C}
$
\item[\bf (4)]
For $1\le \delta<1/\xi^+(C),$
$
A_{\delta,C}\ge \bel{A}_{\delta,C}.
$
\item[\bf (5)]
For any $\lambda\in\mathbb{R},$
$
D_{\lambda,C}\ge \bel{D}_{\lambda,C}.
$
\een
\end{lemma}

To prove it, we will need three ingredients. The first is a simple lemma that says that cutting off the logarithm of an $A_\infty$ weight does not increase its $A_\infty$-characteristic. It is valid in any dimension and applies to both $A_\infty(\rn)$ and the local variant $A_\infty(Q).$
\begin{lemma}
\label{cutoff}
Let $\varphi$ be such that $e^\varphi\in A_\infty.$ For $c,d\in\mathbb{R},$ such that $c< d,$ let
\eq[co]{
\varphi_{c,d}=c\chi^{}_{\{\varphi\le c\}}+ \varphi \chi^{}_{\{c<\varphi<d\}}+d\chi^{}_{\{\varphi\ge d\}}.
}
Then for any cube $J,$
$$
\bav{e^{\varphi_{c,d}-\sav{\varphi_{c,d}}J}}J\le \bav{e^{\varphi-\sav{\varphi}J}}J
$$
and, consequently,
\eq[char1]{
\big[e^{\varphi_{c,d}}\big]_{A_\infty}\le \big[e^{\varphi}\big]_{A_\infty}.
}
\end{lemma}
\begin{proof}
Let $J_1=\{t\in J: \varphi\le c\},$ $J_2=\{t\in J: c<\varphi<d\},$ $J_3=\{t\in J: \varphi\ge d\},$ and for $k=1, 2, 3,$ write $\delta_k=|J_k|/|J|.$ Without loss of generality assume $\delta_k>0$ and put $z_k=\av{\varphi}{J_k},$ $r_k=\av{e^\varphi}{J_k};$ note that $r_k\ge e^{z_k}.$ Then $\av{\varphi_{c,d}}J=c\delta_1+z_2\delta_2+d\delta_3$ and
$$
\bav{e^{\varphi_{c,d}-\sav{\varphi_{c,d}}J}}J
=\delta_1 e^{(1-\delta_1)c-\delta_2z_2-\delta_3d}+\delta_2 r_2 e^{-\delta_1c-\delta_2z_2-\delta_3 d}
+\delta_3 e^{-\delta_1c-\delta_2z_2+(1-\delta_3) d}=:F(c,d).
$$
Fix $\delta_k,$ $z_k,$ and $r_k,$ and consider $F$ as a function of $c$ and $d$ on the domain
$z_1\le c \le z_2 \le \log r_2\le d\le z_3.$ Then
\begin{align*}
F_c(c,d)&=e^{-\delta_1c-\delta_2z_2-\delta_3 d}\left[\delta_1(1-\delta_1)e^c-\delta_1\delta_2 r_2-\delta_1\delta_3 e^d\right]\\
&=e^{-\delta_1c-\delta_2z_2-\delta_3 d}[\delta_1\delta_2(e^c-r_2)+\delta_1\delta_3(e^c-e^d)]\le0,
\end{align*}
where we used the identity $1-\delta_1=\delta_2+\delta_3.$
Similarly
$$
F_d(c,d)=e^{-\delta_1c-\delta_2z_2-\delta_3 d}[\delta_1\delta_3(e^d-e^c)+\delta_2\delta_3(e^d-r_2)]\ge0.
$$
Therefore,
$$
\bav{e^{\varphi_{c,d}-\sav{\varphi_{c,d}}J}}J= F(c,d)\le F(z_1,z_3)=e^{-\sav{\varphi}J}[\delta_1e^{z_1}+\delta_2 r_2+\delta_3 e^{z_2}]\le \bav{e^{\varphi-\sav{\varphi}J}}J.
$$
Inequality~\eqref{char1} for the $A_\infty$-characteristics is immediate.
\end{proof}
The second ingredient is a lemma due to Vasyunin that can be found in~\cite{v1}.
\begin{lemma}[\cite{v1}]
\label{split}
Fix $C\ge1$ and take any $C_1>C.$ Then for every interval $Q$ and every $\varphi$ such that $e^\varphi\in \AC,$ there exists a splitting $Q=Q_-\cup Q_+$ such that the whole line segment with the endpoints $(\av{\varphi}{Q_\pm},\av{e^\varphi}{Q_\pm})$ lies inside $\Omega_{C_1}.$ Moreover, the ratio $|Q_+|/|Q|$ can be chosen to be uniformly \textup(with respect to $\varphi$ and $Q$\textup) separated from $0$ and $1.$ 
\end{lemma}

\subsection{Local convexity and concavity}
The third ingredient in the proof of Lemma~\ref{lemma_induction} is the verification of the local convexity/concavity of our candidates. It requires a certain amount of calculation much of which is routine. The presentation in this part parallels the one in Section~6 of~\cite{sv1}. First, we state a formal definition.
\begin{definition}
A function $G$ is called locally convex on $\Omega_C$ if
\eq[convex]{
G(\alpha  x^-+(1-\alpha)x^+)\le \alpha G(x^-)+(1-\alpha)G(x^+)
}
for all $\alpha\in [0,1]$ and all $x^-, x^+\in \Omega_C$ such that the line segment $[x^-,x^+]$ lies entirely in $\Omega_C.$
Similarly, $G$ is called locally concave on $\Omega_C$ if
\eq[concave]{
G(\alpha  x^-+(1-\alpha)x^+)\ge \alpha G(x^-)+(1-\alpha)G(x^+)
}
for all such $\alpha$ and $x^\pm.$
\end{definition}

Recall the functions defined by various formulas of Section~\ref{bellman}: $u^\pm$ by~\eqref{upm}, $v$ by~\eqref{v},~$b_{p,C}$ by~\eqref{bp}, $b_{2,C}$ by~\eqref{b2}, $B_{2,C}$ by~\eqref{B2}, $b_{1,C}$ by~\eqref{b1domain} and~\eqref{b1}, $A_{\delta,C}$ by~\eqref{Ad}, and $D_{\lambda,C}$ by~\eqref{wfd}, and~\eqref{wf}. As explained in that section, we need to show that $b_{p,C}$ is locally convex for all $p\in[1,2]$ and that $B_{2,C},$ $A_{\delta,C},$ and $D_{\lambda,C}$ are all locally concave.

Before stating a formal result, let us outline what this entails. All our candidates have similar structure: for each candidate $G,$ the domain $\Omega_C$ is the union of one or more subdomains, $\Omega_C=\cup S_k,$ such that $G$ is twice-differentiable and satisfies the homogeneous \ma equation $G_{x_1x_1}G_{x_2x_2}=G_{x_1x_2}^2$ in the interior of each $S_k.$ For such a $G$ showing local convexity (concavity) within each particular $S_k$ amounts to showing that either $G$ is affine in $S_k$ or that $G_{x_2x_2}>0$ ($<0$) in the interior of $S_k$ and $G$ is continuous along every straight-line segment crossing the boundary of $S_k.$ 

To show~\eqref{convex} (respectively, \eqref{concave}) for segments $[x^-,x^+]$ that cross the boundaries between subdomains, we simply need to make sure that $G_{x_2}$ is increasing (respectively, decreasing) across these boundaries. This is because each such boundary is always a non-vertical straight-line segment such that the gradient $\nabla G$ is constant on each side of it (though the two constant vectors are not necessarily the same). As explained in~\cite{sv1}, in this situation it suffices to check the jump in the directional derivative of $G$ in any direction transversal to the boundary, and it is convenient to choose the $x_2$-direction.

We will often encounter the situation when a candidate $G$ has the form
\eq[slope]{
G(x)=m(u)(x-u)+f(u),
}
where $u$ stands for either $u^+(x)$ or $u^-(x)$ and $m$ satisfies the equation 
$$
\xi m'(u)=m(u)-f'(u),
$$
for some differentiable function $f.$ (Here $\xi=\xi^+$ when $u=u^+$ and $\xi=\xi^-$ when $u=u^-.$) Such a function automatically satisfies the \ma equation. Indeed, we compute
$$
u_{x_1}=-\frac1{x_1-u-\xi},\qquad u_{x_2}=\frac{e^{-u}(1-\xi)}{x_1-u-\xi},
$$
\eq[gder]{
G_{x_1}=m-m',\qquad G_{x_2}=m'e^{-u}(1-\xi),
}
and
$$
G_{x_1x_1}\!\!=(m'-m'')u_{x_1},\quad G_{x_2x_2}\!\!=(m''-m')e^{-u}(1-\xi)u_{x_2},\quad G_{x_1x_2}\!\!=(m''-m')e^{-u}(1-\xi)u_{x_1}.
$$
Therefore, 
$
G_{x_1x_1}G_{x_2x_2}=G_{x_1x_2}^2.
$
Moreover, since $u^+_{x_2}<0$ and $u^-_{x_2}>0,$ we have $\sign G_{x_2x_2}=\sign(m'-m'')$ if $u=u^+,$ and $\sign G_{x_2x_2}=\sign(m''-m')$ if $u=u^-.$

\begin{lemma}
\label{concon}
Let $C\ge 1.$
\ben[leftmargin=*]
\item[\bf (1)]
If $p\in(1,2]$ and $C\ge \frac{e^{p-2}}{p-1},$ then $b_{p,C}$ is locally convex in $\Omega_C.$
\item[\bf (2)]
$b_{1,C}$  is locally convex in $\Omega_C.$
\item[\bf (3)]
$B_{2,C}$ is locally concave in $\Omega_C.$
\item[\bf (4)]
For $1\le\delta<1/\xi^+(C),$ $A_{\delta,C}$ is locally concave in $\Omega_C.$
\item[\bf (5)]
For any $\lambda\in\mathbb{R},$ $D_{\lambda,C}$ is locally concave in $\Omega_C.$
\een
\end{lemma}
\begin{proof} As applicable below, write $b$ for $b_{p,C},$ $B$ for 
$B_{2,C},$ $A$ for $A_{\delta,C},$ and $D$ for $D_{\lambda,C}.$
%\ben[leftmargin=*]

%\item
\noindent {\bf (1)}
In this part, write $u$ for $u^+$ and $\xi$ for $\xi^+.$ Clearly, $b$ is continuous on $\Omega_C.$ Moreover, $b$ is of the form~\eqref{slope} with $f(u)=|u|^p$ and $m(u)=\frac1\xi\int_u^\infty f(s)e^{-s/\xi}ds.$ Therefore, $b$ is differentiable in the interior of $\Omega_C$ and twice-differentiable in the interior of each of the two subdomains of $\Omega_C$ separated by the tangent $\ell^+{(0,1)}.$ (Note that $u(x)=0$ if an only if $x\in \ell^+{(0,1)}.$) Since there is no jump in $b_{x_2}$ across this line, we only need to verify that $b_{x_2x_2}>0$ in each (open) subdomain. 

We have $\sign b_{x_2x_2}=\sign(m'-m'').$ Assuming $u\ne0,$ we compute
$$
\xi^2(m'(u)-m''(u))e^{-u/\xi}=\xi f''(u)e^{-u/\xi}-(1-\xi)\int_u^\infty e^{-s/\xi}f''(s)\,ds.
$$
Using $f''(s)=p(p-1)|s|^{p-2}$ and the variable in the integral, we see that $\sign b_{x_2x_2}=\sign F(u/\xi),$ where
$$
F(\mu)=|\mu|^{p-2}e^{-\mu}-(1-\xi)\int_\mu^\infty e^{-t}|t|^{p-2}\,dt.
$$
If $p=2,$ $F(\mu)=\xi e^{-\mu}>0$ and so $b_{x_2x_2}>0$ without any conditions on $C.$  
Assume $p\in (1,2)$ and $C\ge \frac{e^{p-2}}{p-1}=\frac{e^{-(2-p)}}{1-(2-p)};$ then $\xi\ge 2-p.$ 
Differentiation shows that $F$ is decreasing on the interval $(0,\infty)$ and that its minimum on the interval $(-\infty,0)$ is at the point $\mu=(p-2)/\xi.$  Since $F(\mu)\to0$ as $\mu\to\infty,$ the proof will be complete if we verify that $F(\frac{p-2}\xi)\ge0.$ Clearly, $F(\frac{p-2}\xi)$ is increasing in $\xi,$ hence,
$$
F({\textstyle \frac{p-2}\xi})= |{\textstyle \frac{p-2}\xi}|^{p-2}e^{-(p-2)/\xi}-(1-\xi)\int_{(p-2)/\xi}^\infty e^{-t}|t|^{p-2}\,dt\ge e-(p-1)\int_{-1}^\infty e^{-t}|t|^{p-2}\,dt.
$$
Finally, it is easy to check that the last expression is decreasing in $p$ on the interval $p\in(1,2]$ and equals $0$ when $p=2.$ Therefore, $b_{x_2x_2}>0,$ and the proof is complete.
\medskip

%\item
\noindent {\bf (2)}
Since $b$ is an affine (and thus locally convex) function in each of the regions, $\Omega_-,$ $\Omega_0,$ and $\Omega_0,$ we only need to verify that 
$b_{x_2}$ is increasing across the tangent lines $\ell^\pm{(0,1)}$ separating these regions. 
This is elementary:
in $\Omega_\pm,$ $b_{x_2}=0,$ while in $\Omega_0,$ $b_{x_2}=2\,\frac{(1-\xi^-)(1-\xi^+)}{\xi^+-\xi^-}\ge0.$
\medskip

%\item
\noindent {\bf (3)}
The function $B$ is of the form~\eqref{slope} with $m(u)=2(u+\xi),$ $f(u)=u^2,$ $u=u^-,$ and $\xi=\xi^-.$ Thus, it suffices to observe that $\sign B_{x_2x_2}=\sign(m''-m')=\sign(-2)<0.$
\medskip

%\item
\noindent {\bf (4)}
The function $A$ is of the form~\eqref{slope} with $m(u)=\delta e^{\delta u}/(1-\delta \xi),$ $f(u)=e^{\delta u},$ $u=u^+,$ and $\xi=\xi^+.$ Furthermore, we have $\sign A_{x_2x_2}=\sign(m'-m'')=\sign(1-\delta)<0.$ 
\medskip

%\item
\noindent {\bf (5)}
First, observe that $D$ is locally concave in each subdomain. This is trivially true in $\Omega_3(\lambda)$ where $D$ is affine, and in $\Omega_4(\lambda),$ where $D$ is constant. In $\Omega_2(\lambda),$ we compute
$$
D_{x_1x_1}=\frac{e^v}{r^3}\,(x_2-e^\lambda)^2,\quad D_{x_2x_2}=\frac{e^v}{r^3}\,(x_1-\lambda)^2,\quad 
D_{x_1x_1}=-\frac{e^v}{r^3}\,(x_2-e^\lambda)(x_1-\lambda), 
$$
where $r=x_2-e^\lambda-e^v(x_1-\lambda).$ Thus, $D$ satisfies the \ma equation in the interior of $\Omega_2(\lambda).$
It is easy to show that $r<0,$ which means that $D$ is locally concave in $\Omega_2(\lambda).$ (Note that while $D$ has a discontinuity at the point $(\lambda, e^\lambda),$ it is continuous along every straight-line segment contained in $\Omega_C$ and passing through that point.)

In $\Omega_1(\lambda),$ $D$ is of the form~\eqref{slope} with $m(u)=e^{(\xi^+-\xi^--\lambda+u)/\xi^+}/(\xi^+-\xi^-),$ $f(u)=0,$ and $u=u^+.$ Then $\sign D_{x_2x_2}=\sign(m'-m'')=\sgn(\xi^+-1)<0.$

It remains to check the jump in $D_{x_2}$ across the boundaries between subdomains. In $\Omega_4(\lambda),$ $D_{x_2}=0,$ and in $\Omega_3(\lambda),$ $D_{x_2}=-e^{-\lambda}(1-\xi^-)(1-\xi^+)/(\xi^+-\xi^-)^2<0.$ Thus, $D_{x_2}$ is decreasing across the bounding tangent $\ell^+{(\lambda,e^\lambda)}.$ 

In $\Omega_2(\lambda),$ we compute $D_{x_2}=-1/(e^v(1-v+\lambda)-e^\lambda).$ On the the boundary separating $\Omega_2(\lambda)$ from both $\Omega_3(\lambda)$ and $\Omega_1(\lambda),$ we have $v(x)=\lambda+\xi^--\xi^+$ (see Figure~\ref{f3}), thus, $D_{x_2}=e^{-\lambda-\xi^-+\xi^+}/(e^{-\xi^-+\xi^+}-1+\xi^--\xi^+)$ on this line. Since
$
e^{-\xi^-+\xi^+}=\frac{1-\xi^-}{1-\xi^+},
$
 $D_{x_2}=e^{-\lambda-\xi^-+\xi^+}(1-\xi^+)/(\xi^+(\xi^+-\xi^-))>0.$
Thus $D_{x_2}$ is decreasing from $\Omega_2(\lambda)$ into $\Omega_3(\lambda).$

Lastly, in $\Omega_1(\lambda),$ $D_{x_2}$ is given by~\eqref{gder}: 
$$
D_{x_2}=m'(u^+)e^{-u^+}(1-\xi^+)=\frac{e^{(\xi^+-\xi^--\lambda+u^+)/\xi^+}}{\xi^+(\xi^+-\xi^-)}e^{-u^+}(1-\xi^+).
$$
On the tangent $\ell^+{(\lambda+\xi^--\xi^+,e^{\lambda+\xi^--\xi^+})},$ we have $u^+=\lambda+\xi^--\xi^+.$ Hence, on this line,
$
D_{x_2}=e^{-\lambda+\xi^+-\xi^-}(1-\xi^+)/(\xi^+(\xi^+-\xi^-)),
$
which means that the jump in $D_{x_2}$ from $\Omega_2(\lambda)$ into $\Omega_1(\lambda)$ is $0.$ This completes the proof.
%\een
%\vspace{-.5cm}
\end{proof}

\subsection{The induction}
We are now ready to prove Lemma~\ref{lemma_induction}.
\begin{proof}[Proof of Lemma~\ref{lemma_induction}]
We only prove statement (1), as the proofs of the other statements are fully analogous. 

Fix an interval $Q$ and take $\varphi\in L^\infty(Q)$ such that $e^\varphi\in\AC.$ Take any $C_1>C.$ We first build a special collection of subintervals of $Q.$ Let $D_0(Q)=\{Q\}.$ Let $Q=Q_-\cup Q_+$ be the splitting provided by Lemma~\ref{split} and write $D_1(Q)=\{Q_-,Q_+\}.$ Note that $e^\varphi\in A_\infty^C(Q_\pm)$ so the same lemma can be applied separately to $Q_-$ and $Q_+,$ which defines $D_1(Q_-)$ and $D_1(Q_+).$ For $n\ge 2,$ let $D_n(Q)=D_{n-1}(Q_-)\cup D_{n-1}(Q_+)$ and $D(Q)=\cup_{n\ge 1}D_n(Q).$ Observe that $\max_{J\in D_n(Q)}|J|\to 0$ as $n\to\infty.$

For all $J\in D(Q),$ let $x^J=(\av{\varphi}J,\av{e^\varphi}J)$ and consider the following sequence of step functions taking values in 
$\Omega_C:$
$$
x^{(n)}(t)=\sum_{J\in D_n(Q)} x^J\chi^{}_{J}(t).
$$
By the Lebesgue differentiation theorem, $x_n\to (\varphi,e^\varphi)$ {\it a.e.} on $Q.$ 

By Lemma~\ref{split}, for any $J\in D(Q)$ the line segment connecting the points $x^{J_-}$ and $x^{J_+}$ lies entirely in $\Omega_{C_1}.$ By Lemma~\ref{concon}, $b_{p,C_1}$ is locally convex in $\Omega_{C_1},$ i.e., it satisfies~\eqref{convex} with $C$ replaced with $C_1.$ Using these two facts repeatedly, we obtain
\begin{align*}
b_{p,C_1}(\av{\varphi}Q,\av{e^\varphi}Q)&=b_{p,C_1}(x^Q)\le \frac{|Q_-|}{|Q|}b_{p,C_1}(x^{Q_-})+\frac{|Q_+|}{|Q|}b_{p,C_1}(x^{Q_+})\\
&\le \frac1{|Q|}\,\sum_{J\in D_n(Q)} |J| b_{p,C_1}(x^J)=\frac1{|Q|}\,\int_Q b_{p,C_1}(x^{(n)}(t))\,dt.
\end{align*}
Note that $\{x^{(n)}\}$ is a uniformly bounded sequence and $b_{p,C_1}$ is continuous in $\Omega_{C}.$ Therefore, the dominated convergence theorem applies, and, since $b_{p,C_1}(s,e^s)=|s|^p$ for any $s,$ we have 
\eq[bound]{
b_{p,C_1}(\av{\varphi}Q,\av{e^\varphi}Q)\le \frac1{|Q|}\,\int_Q b_{p,C_1}(\varphi(t),e^{\varphi(t)})\,dt=\av{|\varphi|^p}Q.
}
Now take any, not necessarily bounded $\varphi\in E_{x,C,Q}.$ For $c,d\in\mathbb{R}$ such that $c<d,$ let $\varphi_{c,d}$ be defined by~\eqref{co}. Then $\varphi_{c,d}\in L^\infty(Q),$ and by Lemma~\ref{cutoff}, $e^{\varphi_{c,d}}\in \AC;$ thus,~\eqref{bound} applies:
$$
b_{p,C_1}(\av{\varphi_{c,d}}Q,\av{e^{\varphi_{c,d}}}Q)\le \av{|\varphi_{c,d}|^p}Q.
$$
Taking the limit first as $c\to-\infty$ and then as $d\to\infty,$ and using the monotone convergence theorem, we obtain
$$
b_{p,C_1}(x)\le \av{|\varphi|^p}Q.
$$
Taking infimum over all $\varphi\in E_{x,C,Q}$ now gives 
$$
b_{p,C_1}(x)\le \bel{b}_{p,C}(x),
$$
and it remains only to observe that $b_{p,C_1}(x)$ is continuous in $C_1$ for each $x\in\Omega_C$ and take limit as $C_1\to C.$
\end{proof}

\section{Optimizers and converse inequalities}
\label{optimizers}
The main result of this section is the following converse to Lemma~\ref{lemma_induction}, which will conclude the proof of Theorem~\ref{all_bellman}. Note that unlike in Lemma~\ref{lemma_induction}, there is no restriction on $C$ in the first statement.

\begin{lemma}
\label{main_opt0}
Let $C\ge1.$
\ben[leftmargin=*]
\item[\bf (1)]
If $p\in[1,2],$ then 
$
b_{p,C}\le \bel{b}_{p,C}.
$
\item[\bf (2)]
$
B_{2,C}\ge \bel{B}_{2,C}
$
\item[\bf (3)]
For $1\le \delta<1/\xi^+(C),$
$
A_{\delta,C}\ge \bel{A}_{\delta,C}.
$
\item[\bf (4)]
For any $\lambda\in\mathbb{R},$
$
D_{\lambda,C}\ge \bel{D}_{\lambda,C}.
$
\een
\end{lemma}

We can assume $C>1.$ Let $Q=(0,1).$  If $G$ is one of the candidates in Lemma~\ref{main_opt0} and $x\in\Omega_C,$ we say that a function $\varphi_x$ on $Q$ is an optimizer for $G(x)$ if
\eq[optimizer1]{
\varphi_x\in E_{x,C,Q}\quad\text{and}\quad \av{f(\varphi_x)}{Q}=G(x),
}
where $f$ is the function that appears in the definition of the corresponding Bellman function: $f(t)=|t|^p$ for $b_{p,C};$ $f(t)=t^2$ for $B_{2,C};$ $f(t)=e^{\delta t}$ for $A_{\delta,C};$ and $f(t)=\chi^{}_{[\lambda,\infty)}(t)$ for $D_{\lambda,C}.$ Clearly, presenting such an optimizer for each candidate $G$ and each $x\in\Omega_C$ will prove Lemma~\ref{main_opt0}.

We now do just that. Verifying that our claimed optimizers have the required integral averages on $Q$ is straightforward, as they are explicitly designed to have that property (we do not go into details of how such optimizers arise and instead refer the reader to papers~\cite{sv1} and~\cite{iosvz2}). Verifying that each optimizer $\varphi_x$ satisfies $e^{\varphi_x}\in\AC$ is more subtle. Note that this is equivalent to showing that for any interval $J\subset Q$  the point $(\av{\varphi_x}J,\av{e^{\varphi_x}}J)$ lies in $\Omega_C.$ To demonstrate that fact in each case, we closely follow the geometric arguments from Section~7 of~\cite{sv1}, where a similar array of optimizers was shown to be in BMO with a fixed norm.

Recall functions $u^\pm$ given by~\eqref{upm} and function $v$ given by~\eqref{v}.
Our first optimizer will service both $b_{p,C}$ for $1<p\le 2$ and $A_{\delta,C}.$ 
For all $x\in \Omega_C,$ let
\eq[opt+]{
\varphi^+_x(t)=u^+ +\xi^+\log\big({\textstyle\frac{\alpha^+}t}\big)\chi^{}_{(0,\alpha^+)}(t),\quad \text{where}\quad \alpha^+=\frac{x_1-u^+}{\xi^+}.
}

The optimizer for $B_{2,C}$ is obtained by replacing in~\eqref{opt+} $u^+$ and $\xi^+$ with $u^-$ and $\xi^-,$ respectively. Specifically, for $x\in\Omega_C,$ let
\eq[opt-]{
\varphi^-_x(t)=u^-+\xi^-\log\big({\textstyle\frac{\alpha^-}t}\big)\chi^{}_{(0,\alpha^-)}(t), \quad \text{where}\quad \alpha^-=\frac{x_1-u^-}{\xi^-}. 
}

To define the optimizers for $b_{1,C}$ we need to consider three cases.
For $x\in \Omega_+,$ let
\eq[psi+]{
\psi_x=u^+ +(\xi^+-\xi^-)\chi_{(0,\beta^+)},\quad \text{where}\quad \beta^+=\frac{x_1-u^+}{\xi^+-\xi^-}.
}
For $x\in \Omega_-,$ let 
\eq[psi-]{
\psi_x=u^- -(\xi^+-\xi^-)\chi_{(1-\beta^-,1)}, \quad \text{where}\quad \beta^-=\frac{u^--x_1}{\xi^+-\xi^-}.
}
For $x\in \Omega_0,$ let 
\eq[psi0]{
\psi_x=(\xi^+-\xi^-)\big(\chi^{}_{(0,\gamma^+)}-\chi^{}_{(1-\gamma^-,1)}\big),
}
where
\eq[gamma]{
\gamma^\pm=\frac{(x_2-1)(1-\xi^-)(1-\xi^+)-x_1(1-\xi^\pm)}{(\xi^+-\xi^-)^2}.
}

Defining the optimizer for our most sophisticated candidate, $D_{\lambda,C},$ technically requires the consideration of four cases. Fortunately, we can reuse some of the optimizers already defined.

For $x\in\Omega_3(\lambda)\cup\Omega_4(\lambda),$ let
\eq[eta34]{
\eta_x=\lambda+\psi_{(x_1-\lambda,x_2e^{-\lambda})},
}
where $\psi_{(x_1-\lambda,x_2e^{-\lambda})}$ is given by~\eqref{psi0}-\eqref{gamma} if $x\in\Omega_3(\lambda)$ and by~\eqref{psi+} if $x\in\Omega_4(\lambda),$ in both cases with $x$ replaced by $(x_1-\lambda,x_2e^{-\lambda}).$ (Note that $\Omega_3(\lambda)$ and $\Omega_4(\lambda)$ are the images under the transformation $x\mapsto (x_1+\lambda,x_2e^{\lambda})$ of $\Omega_0$ and $\Omega_+,$ respectively.)

For $x\in\Omega_2(\lambda),$ let
\eq[eta2]{
\eta_x=\lambda\chi^{}_{(0,\mu)}+v\chi^{}_{(\mu,1)},\quad\text{where}\quad \mu=\frac{x_1-v}{\lambda-v}
}

For $x\in\Omega_1(\lambda),$ let
\eq[eta1]{
\eta_x=\lambda \chi^{}_{(0,\tau\beta^+)}+(\lambda+\xi^--\xi^+)\chi^{}_{(\tau\beta^+,\tau\alpha^+)}+
\varphi^+_x\,\chi^{}_{(\tau\alpha^+,1)},
}
where
\eq[eta1.1]{
\tau=e^{(u^++\xi^+-\xi^--\lambda)/\xi^+},\quad \beta^+=\frac{x_1-u^+}{\xi^+-\xi^-},
}
and $\varphi^+_x$ is given by~\eqref{opt+}.

We can now prove the following lemma, which will immediately imply Lemma~\ref{main_opt0}.

\begin{lemma}
\label{main_opt}
~
\ben[leftmargin=*]
\item[\bf (1)]
The function $\varphi^+_x$ given by~\eqref{opt+} is an optimizer for $b_{p,C}(x),$ $1<p\le 2,$ and for $A_{\delta,C}(x).$
\item[\bf (2)]
The function $\varphi^-_x$ given by~\eqref{opt-} is an optimizer for $B_{2,C}(x).$
\item[\bf (3)]
The function $\psi_x$ given by~\eqref{psi+},~\eqref{psi-}, and~\eqref{psi0}-\eqref{gamma}, is an optimizer for $b_{1,C}(x).$
\item[\bf (4)]
The function $\eta_x$ given by~\eqref{eta34},~\eqref{eta2}, and~\eqref{eta1}-\eqref{eta1.1}, is an optimizer for $D_{\lambda,C}(x).$
\een
\end{lemma}
\begin{proof}
%\begin{enumerate}[leftmargin=*]
%\item

\noindent {\bf (1)} For any differentiable and appropriately integrable function $f$ we have
\begin{align*}
\av{f(\varphi^+_x)}Q&=f(u^+)(1-\alpha^+)+\int_0^{\alpha^+}\!\! f\big(u^++\xi^+\log(\alpha^+/t)\big)\,dt\\
&=f(u^+)(1-\alpha^+)+\frac1{\xi^+}\,e^{u^+/\xi^+}\alpha^+\int_{u^+}^\infty f(s) e^{-s/\xi^+}\,ds\\
&=\frac1{\xi^+}\,e^{u^+/\xi^+}\Big[\int_{u^+}^\infty f'(s) e^{-s/\xi^+}\,ds\Big](x_1-u^+)+f(u^+),
\end{align*}
with the first equality due to a change of variable and the second, to integration by parts.

Now, as the reader can check, $f(s)=s$ gives $\av{\varphi^+_x}Q=x_1;$ $f(s)=e^s$ gives $\av{e^{\varphi^+_x}}Q=x_2;$ $f(s)=|s|^p$ gives
$\av{|\varphi^+_x|^p}Q=b_{p,C}(x);$ and $f(s)=e^{\delta s}$ gives $\av{e^{\delta\varphi^+_x}}Q=A_{\delta,C}(x),$ if $\delta<1/\xi^+,$ and $\av{e^{\delta\varphi^+_x}}Q=\infty,$ if $\delta\ge 1/\xi^+$ (unless $x_2=e^{x_1},$ in which case $u^+=x_1$ and $\alpha^+=0,$ so $\av{e^{\delta\varphi^+_x}}Q=e^{\delta x_1}$).

Observe that $\varphi_x^+$ is the cut-off from below, at height $u^+,$ of the function $u^++\xi^+\log(\alpha^+)+\xi^+\varphi_0(t),$ where $\varphi_0(t)=\log(1/t)$ from Lemma~\ref{phi0}. By~\eqref{char} from the proof of that lemma, $[e^{\xi^+\varphi_0}]_{A_\infty}=e^{-\xi^+}/(1-\xi^+)=C.$ Therefore, by Lemma~\ref{cutoff}, $e^{\varphi_x^+}\in \AC.$  
\medskip

%\item
\noindent {\bf (2)} This is a repetition of the argument just given for $b_{2,C},$ with $\xi^-,$ $u^-,$ and $\alpha^-$ replacing $\xi^+,$ $u^+,$ and $\alpha^+,$ respectively.
\medskip

%\item
\noindent {\bf (3)}
The cases $x\in \Omega_+$ and $x\in \Omega_-$ are symmetric, so we only consider the first one. Here we have $\psi_x\ge0,$ so $\av{\psi_x}Q=\av{|\psi_x|}Q=u^++(\xi^+-\xi^-)\beta^+=x_1.$ Furthermore, 
$
\av{e^{\psi_x}}Q=e^{u^+}(1-\beta^++e^{\xi^+-\xi^-}\beta^+)=x_2,
$
where we used the identity $e^{\xi^+-\xi^-}=(1-\xi^-)/(1-\xi^+).$ 

To see that $\psi_x\in \AC,$ note that $\psi_x$ is a concatenation of two constant functions $u^++\xi^+-\xi^-$ and 
$u^+.$ Therefore, for any interval $J\subset Q$ the point $(\av{\psi_x}J,\av{e^{\psi_x}}J)$ is a convex combination of the two points where the tangent $\ell^+(x)$ intersects $\Gamma_1,$ hence it lies in $\Omega_C.$

Now, assume that $x\in \Omega_0.$ Then $\av{\psi_x}Q=(\xi^+-\xi^-)(\gamma^+-\gamma^-)=x_1,$ $\av{|\psi_x|}Q=(\xi^+-\xi^-)(\gamma^++\gamma^-)=b_{1,C}(x),$ and
$
\av{e^{\psi_x}}Q=1+\gamma^+(e^{\xi^+-\xi^-}-1)-\gamma^-(e^{\xi^--\xi^+}-1)=x_2.
$

To show that $\psi_x\in \AC,$ draw a line $\ell$ through $x$ so that it intersects both $\ell^-{(0,1)}$ and $\ell^+{(0,1)};$ call the points of intersection $x^-$ and $x^+,$ respectively. Note that it is always possible to draw $\ell$ so that the segment $[x^-,x^+]\subset \Omega_0,$ so assume that is the case. Let $\delta =\frac{x_1-x_1^-}{x_1^+-x_1^-};$ it is easy to verify that $\delta\in[\gamma^+,1-\gamma^-].$ Take an interval $J=(c,d)\subset Q.$ If $J\subset (0,\delta),$ then the point $U:=(\av{\psi_x}J,\av{e^{\psi_x}}J)$ is a convex combination of the points $(0,1)$ and $(\xi^+-\xi^-,e^{\xi^+-\xi^-})$ and thus in $\Omega_C;$ and similarly for the case $J\subset(\delta,1).$ Assume now that $c\le \gamma^+<1-\gamma^-\le d.$ In this case, $U$ is a convex combination of $V:=(\av{\psi_x}{(c,\delta)},\av{e^{\psi_x}}{(c,\delta)})$ and $W:=(\av{\psi_x}{(\delta,d)},\av{e^{\psi_x}}{(\delta,d)}).$ Moreover, $V$ lies on $\ell^+{(0,1)},$ below $x^+;$ similarly, $W$ is on $\ell^-{(0,1)},$ below $x^-.$ We conclude that $[V,W]\subset \Omega_0$ and so $U\in \Omega_0\subset\Omega_C.$
\medskip

%\item 
\noindent {\bf (4)}
Observe that $\eta_x$ is an optimizer for $D_{0,C}(x)$ if and only if $\lambda+\eta_{(x_1+\lambda,x_2e^{\lambda})}$ is an optimizer for $D_{\lambda,C}(x_1+\lambda,x_2e^\lambda).$ Thus, it suffices to consider the case $\lambda=0.$ 

Setting $\lambda=0$ gives $\Omega_4(0)=\Omega_+,$ $\Omega_3(0)=\Omega_0,$ and $\eta_x=\psi_x$ if $x\in \Omega_3(0)\cup\Omega_4(0).$ We have already verified that $\av{\psi_x}Q=x_1,$ $\av{e^{\psi_x}}Q=x_2,$ and $e^{\psi_x}\in \AC.$ If $x\in \Omega_+,$ then $\psi_x\ge0,$ so $|\{\psi_x\ge0\}|=1=D_{0,C}(x).$ If $x\in \Omega_0,$ then 
$|\{\psi_x\ge0\}|=1-\gamma^-=D_{0,C}(x).$

Assume $x\in \Omega_2(0).$ Since $\lambda=0,$ the defining equation for $v$ becomes $x_1/v=x_2e^{-v}$ and we have $\mu=D_{0,C}(x)=1-x_1/v.$ Now, $\av{\eta_x}Q=v(1-\mu)=x_1,$ 
$\av{e^{\eta_x}}Q=\mu+e^{v}(1-\mu)=x_2,$ and $|\{\eta_x\ge0\}|=\mu=D_{0,C}(x).$ To show that $e^{\eta_x}\in \AC,$ note that $\eta_x$ is a concatenation of the constant functions $\lambda$ and $v,$ and, thus, for any interval $J\subset Q$ the point 
$(\av{\eta_x}J,\av{e^{\eta_x}}J)$ lies on the line segment connecting the points $(\lambda,e^\lambda)$ and $(v,e^v).$ In turn, this segment lies in $\Omega_2(0)\subset \Omega_C.$

Finally, assume that $x\in\Omega_1(0).$ We have
$$
\eta_x=(\xi^--\xi^+)\chi^{}_{(\tau\beta^+,\tau\alpha^+)}+
(u^+ +\xi^+\log(\alpha^+/t))\chi^{}_{(\tau\alpha^+,\alpha^+)}(t)+u^+\chi^{}_{(\alpha^+,1)},
$$
so,
\eq[e1]{
\begin{split}
\av{\eta_x}Q&=(\xi^-\!\!-\xi^+)\tau(\alpha^+\!\!-\beta^+)+\xi^+\int_{\tau\alpha^+}^{\alpha^+}\log(\alpha^+/t)\,dt
+u^+(1-\tau\alpha^+)\\
&=(\xi^-\!\!-\xi^+)\tau(\alpha^+\!\!-\beta^+)+\alpha^+\xi^+(1+\tau(\log\tau-1))+u^+(1-\tau\alpha^+)\\
&=\tau(-\xi^+\alpha^+-(\xi^--\xi^+)\beta^+))+\alpha^+\xi^++u^+=x_1,
\end{split}
}
\begin{align}
\label{e2}
\notag \av{e^{\eta_x}}Q&=\tau\beta^++e^{\xi^-\!\!-\xi^+}\tau(\alpha^+\!\!-\beta^+)+
e^{u^+}\Big(1-\alpha^++\int_{\tau\alpha^+}^{\alpha^+}(\alpha^+/t)^{\xi^+}dt\Big)\\
&=\tau\Big(\beta^++e^{\xi^-\!\!-\xi^+}(\alpha^+\!\!-\beta^+)+\alpha^+\frac{e^{u^+}\tau^{-\xi^+}}{1-\xi^+}\Big)
+e^{u^+}\Big(1-\alpha^++\frac{\alpha^+}{1-\xi^+}\Big)\\
\notag &=\tau\alpha^+\Big(\frac{\xi^+}{\xi^+-\xi^-}\Big(1-\frac{1-\xi^+}{1-\xi^-}\Big)+\frac{\xi^+}{1-\xi^-}\Big)
+e^{u^+}\Big(1+\frac{x_1-u^+}{1-\xi^+}\Big)=x_2
\end{align}

(where the second to last equality uses the identity $e^{u^+}\tau^{-\xi^+}=e^{\xi^--\xi^+}=\frac{1-\xi^+}{1-\xi^-}$), and
$$
|\{\eta_x\ge0\}|=|(0,\tau\beta^+)|=\frac{x_1-u}{\xi^+-\xi^-}\,e^{(u^++\xi^+-\xi^-)/\xi^+}=D_{0,C}(x).
$$ 

To show that $e^{\eta_x}\in \AC,$ note that $\eta_x$ is the cut-off from below, at height $u^+,$ of the function
$$
\zeta(t):=(\xi^--\xi^+)\chi^{}_{(\tau\beta^+,\tau\alpha^+)}(t)+
(u^+ +\xi^+\log(\alpha^+/t))\chi^{}_{(\tau\alpha^+,1)}(t),
$$
so by Lemma~\ref{cutoff} it is enough to show that $e^{\zeta}\in \AC.$ Take an interval $J=(c,d)\subset Q$ and let $U=(\av{\zeta}J,\av{e^{\zeta}}J).$ If $J\subset(0,\tau\alpha^+),$ then $U$ is a convex combination of the points $(0,1)$ and $(\xi^--\xi^+,e^{\xi^--\xi^+})$ and thus lies on the tangent $\ell^-{(0,1)}.$ 

If $J\subset (\tau\beta^+,1),$ then we can write $U=(\av{\omega}J,\av{e^{\omega}}J)$ where $\omega$ is the cut-off from above, at height $\xi^--\xi^+,$ of the function $u^++\xi^+\log(\alpha^+/t).$ That function was treated in part~(1) above. Thus, again by Lemma~\ref{cutoff}, $U\in\Omega_C.$ 

Lastly, assume that $c\le \tau\beta^+ <\tau\alpha^+\le d.$ Write $E=(\av{\zeta}{(0,c)},\av{e^{\zeta}}{(0,c)})$ and $F=
(\av{\zeta}{(0,d)},\av{e^{\zeta}}{(0,d)}).$ Then $E=(0,1),$ while a direct computation similar to~\eqref{e1}-\eqref{e2} gives
$$
F=\Big(u^++\xi^++\xi^+\log\big(\textstyle\frac{\alpha^+}d\big),\frac1{1-\xi^+}e^{u^+}(\frac{\alpha^+} d)^{\xi^+}\Big)=:(F_1,F_2).
$$
Note that $F_2e^{-F_1}=C$ and $F_1\le u^++\xi^++\xi^+\log(1/\tau)=\xi^-.$ Therefore, $F$ lies on the upper boundary 
$\Gamma_C$ and to the left of the point $(\xi^-,Ce^{\xi^-})$ at which the line $\ell^-{(0,1)}$ is tangent to $\Gamma_C.$ This means that the ray from $E$ through $F$ first exits $\Omega_C$ and then re-enters it at $F.$ Since $F$ is a convex combination of $E$ and $U,$ $U$ lies on the same ray, to the left of $F$ and, thus, in $\Omega_C.$
%\een
The proof is complete.
\end{proof}

\end{document}